\documentclass[oneside, book,11 pt]{amsart}
\usepackage{color}
\usepackage{cancel}
\usepackage{amsmath}
\usepackage{amssymb}
\usepackage[normalem]{ulem}
\usepackage{amsfonts}
\usepackage{bbm}
\usepackage{amsthm}
\usepackage{enumerate}
\usepackage[mathscr]{eucal}
\usepackage{graphicx}
\usepackage{pstricks}
\usepackage[active]{srcltx}
\usepackage{fancyhdr}
\usepackage{verbatim}
\usepackage{fullpage}
\usepackage{hyperref}
\usepackage{mathtools}
\hypersetup{
colorlinks = true, 
urlcolor = blue, 
linkcolor = red, 
citecolor = blue 
}
\usepackage{caption}
\usepackage{float}

\newtheorem{thm}{Theorem}[section]
\newtheorem*{thm*}{Theorem}
\newtheorem{lemma}[thm]{Lemma}
\newtheorem{corollary}[thm]{Corollary}
\newtheorem{prop}[thm]{Proposition}
\newtheorem*{prop*}{Proposition}

\numberwithin{equation}{section}

\title{Mixing Properties for Hom-shifts and the Distance Between Walks on Associated Graphs}
\author{
Nishant Chandgotia 
\and
Brian Marcus
}\address{School of Mathematical Sciences\\
Tel Aviv University,
Israel}
\email {nishant.chandgotia@gmail.com}

\address{
Department of Mathematics\\
University of British Columbia,
Canada }
\email{marcus@math.ubc.ca}
\def\F{{\mathcal F}}
\def\A{{\mathcal A}}
\def\N{\mathbb N}

\newcommand{\Z}{\mathbb{Z}}
\def \B{\mathcal B}

\def \E{\mathbb E}
\def \R{\mathbb R}
\def \G{\mathcal G}
\def \V{\mathcal V}
\def \E1{\mathcal E}

\def \L{\mathcal L}

\def\t{\tilde}
\def \m{\vec}
\def \H{\mathcal H}

\def\mi{{\vec{i}}}
\def\mj{{\vec{j}}}
\def\l{{\langle}}
\def\r{{\rangle}}
\newcommand{\Mod}[1]{\ (\text{mod}\ #1)}
\begin{document}
\maketitle
\begin{abstract} Let $\H$ be a finite connected undirected graph and $\H^2_{walk}$ be the graph of bi-infinite walks on $\H$; two such walks $\{x_i\}_{i\in \Z}$ and $\{y_i\}_{i \in \Z}$ are said to be adjacent if $x_i$ is adjacent to $y_i$ for all $i \in \Z$. We consider the question: Given a graph $\H$ when is the diameter (with respect to the graph metric) of $\H^2_{walk}$ finite? Such questions arise while studying mixing properties of hom-shifts (shift spaces which arise as the space of graph homomorphisms from the Cayley graph of $\Z^d$ with respect to the standard generators to $\H$) and are the subject of this paper. 
\end{abstract}

\section{Introduction}

Let $\A$ be a finite set called the alphabet. A shape is a finite subset of $\Z^d$ and a pattern is a function from a shape to the alphabet $\A$. Given a finite set of patterns $\F$ called a forbidden list, a shift of finite type (SFT) $X_\F\subset \A^{\Z^d}$ is the set of configurations in which patterns from $\F$ and their translates do not appear. There is a natural topology on $X_\F$ coming from the product of the discrete topology on $\A$ making it a compact metrisable space; $\Z^d$ acts on it by translation of configurations making it a dynamical system. The study of SFTs for $d\geq 2$ is rife with numerous undecidability issues. It is not even decidable if an SFT is non-empty \cite{bergerundecidable}. It follows immediately that most non-trivial properties of SFTs are undecidable (Proposition \ref{prop: undecidability of transitivity/block-gluing}). In this paper we study an important class of SFTs called hom-shifts, for which, a priori many such issues do not arise.

By $\Z^d$ we will mean both the group and its Cayley graph with respect to standard generators. Given any SFT $X_\F$, we can assume by a standard recoding argument that $X_\F$ is in fact a nearest neighbour SFT (possibly for a different alphabet $\A$), meaning $\F$ consists of patterns on edges and vertices of $\Z^d$. Let $Hom(\G,\H)$ denote the set of all graph homomorphisms from $\G$ to $\H$. An SFT $X$ is called a hom-shift if $X=Hom(\Z^d, \H)$ for some graph $\H$; it is denoted by $X^d_\H$. Alternatively, a hom-shift can be described as a nearest neighbour SFT which is `symmetric' and `isotropic', that is, if $v,w\in \A$ are forbidden to sit next to each other in some coordinate direction, then they are forbidden to sit next to each other in all coordinate directions. It follows that a hom-shift $X^d_\H$ is non-empty if and only if $\H$ has at least one edge. An introduction to SFTs and hom-shifts can be found in Section \ref{Section:SFTs, Hom-Shifts and Mixing Conditions}.

Many important SFTs arise as hom-shifts like the hard square shift and the $n$-coloured chessboard. In this paper we study certain mixing properties of hom-shifts: topological mixing, block-gluing and strong irreducibility and relate them to some natural questions in graph theory. The mixing conditions studied in this paper are introduced in Section \ref{Section: Some Mixing Conditions for Hom-Shifts}. For further background consider \cite{boyle2010multidimensional}.

An SFT $X$ is said to be topologically mixing (or just mixing) if any two patterns appearing in $X$ can coappear in a configuration in $X$ provided the corresponding shapes are far enough apart (the distance depending on the patterns). Clearly, a hom-shift $X^d_\H$ is not mixing if $\H$ is bipartite; the pattern on any partite class of $\Z^d$ is mapped into a partite class of $\H$. It turns out that this is essentially the only obstruction. We prove in Proposition \ref{prop: transitivity hom-shifts} that a hom-shift $X^d_\H$ is mixing if and only if $\H$ is a connected undirected graph which is not bipartite; further if $\H$ is bipartite then it still satisfies a similar mixing condition but we may need to translate one of the two patterns by a unit coordinate vector.  In the heart of the analysis is the following simple idea: We say that two finite walks, $\{v_i\}_{i=1}^n$ and $\{w_i\}_{i=1}^n$ are adjacent if $v_i$ is adjacent to $w_i$ for all $i$. We show that for all $n$ and finite connected graphs $\H$, the graph of finite walks of length $n$ is connected. 

However we find that the diameter of the graph of finite walks on a graph $\H$ of length $n$ might increase with $n$. Whether the diameter remains bounded or not relates to another important mixing property called the phased block-gluing property: We say that an SFT $X$ is block-gluing if there is an $n\in \N$ such that any two patterns on rectangular shapes in $X$ can coexist in a configuration in $X$ provided that they are separated by distance $n$. Strong irreducibility (SI) is a similar (though a much stronger) mixing property where there is no restriction on the shape of the patterns.

Again we observe that if the graph $\H$ is bipartite then $X^d_\H$ is neither block-gluing nor SI. To remedy the situation we introduce the phased block-gluing and the phased SI properties in Section \ref{section: The Block-Gluing Property for Hom-Shifts} which are similar to the usual block-gluing and SI properties but there is a fixed finite set $S\subset \Z^d$ by elements of which we are allowed to translate one of the two patterns. We prove in Propositions \ref{prop:fromshifttowalk} and \ref{prop:Phaed SI} that if $\H$ is not bipartite and $X^d_\H$ is phased block-gluing/phased SI then it is block-gluing/SI respectively. Further if $\H$ is bipartite and $X^d_\H$ is phased block-gluing/phased SI then the set $S$ can be chosen to be the origin and any of the coordinate unit vectors. This is done by relating the mixing conditions with some natural graph theoretic questions. 

The study of the phased block-gluing property for the $d$-dimensional shift space $X^d_\H$ relates to a natural graph structure on $X^{d-1}_\H$: $x, y\in X^{d-1}_\H$ are said to be adjacent if $x_{\mi}$ is adjacent to $y_{\mi}$ for all $\mi$ in $\Z^{d-1}$. Denote the graph thus obtained by $\H^d_{walk}$. In Proposition \ref{prop:fromshifttowalk} we prove that $X^d_{\H}$ is phased block-gluing if and only if the diameter of $\H^d_{walk}$ is finite.

It can be proved using the ideas of graph folding in \cite{Nowwinkler,brightwell2000gibbs} that if $\H$ is a tree then the space $X^d_\H$ is phased SI. This turns out to be a characterisation for the phased SI property for a large class of graphs: A graph is called four-cycle free if it is connected, it has no self-loops and the four-cycle, $C_4$ is not a subgraph. In Section \ref{section:Phased Mixing Properties for $C_4$-Hom-Free Graphs} we prove for four-cycle free graphs $\H$ that $X^d_\H$ is phased block-gluing/phased SI if and only if $\H$ is a tree. Surprisingly the proof goes via lifts to the universal cover of the graph; in fact following \cite{Marcinfourcyclefree2014} we prove the results for a more general class of graphs called the four-cycle hom-free graphs (defined in Section \ref{section:Phased Mixing Properties for $C_4$-Hom-Free Graphs}). In Subsection \ref{subsection:Why is Four-Cycle Hom-Free Condition Necessary} we discuss why this characterisation fails when the four-cycle hom-free restriction is removed. The paper concludes with a long list of open questions (Section \ref{section: open question}).

Let us summarise. Results regarding decidability among hom-shifts and shifts of finite type are Proposition \ref{Proposition: decidability to periodic dense conjugacy}, Corollary \ref{corollary: decidability conjugate to hom-shift} and Proposition \ref{prop: undecidability of transitivity/block-gluing}; in Subsections \ref{subsection:Decidability of Block-Gluing Distance} and \ref{subsection: homshift sft conjugate} we mention some related open questions. In the proof of Proposition \ref{prop: transitivity hom-shifts}  and in Proposition \ref{prop:fromshifttowalk} we reformulate transitivity, mixing and block-gluing in terms of walks on graphs. Proposition \ref{prop: transitivity hom-shifts} gives necessary and sufficient conditions for transitivity and mixing. Section \ref{section:Phased Mixing Properties for $C_4$-Hom-Free Graphs} discusses the mixing properties for hom-shifts where the corresponding graph is four-cycle hom-free.

We end the introduction with the question which is the cornerstone for this line of research; this we are unable to address. For a more detailed discussion, look at Subsection \ref{subsection:Decidability of Block-Gluing Distance}.
\\

\noindent\textbf{Question:} Is it decidable whether a hom-shift is SI/block-gluing?
\section{SFTs and Hom-Shifts}\label{Section:SFTs, Hom-Shifts and Mixing Conditions}

Let $\A$ be a finite set which we refer to as the \emph{alphabet} with the discrete topology; we give the set $\A^{\Z^d}$ the product topology making it a compact metrizable space. By $\Z^d$ we will mean both the Cayley graph of $\Z^d$ with respect to standard generators and the group. The elements of $\A^{\Z^d}$ are called \emph{configurations} while elements of $\A^{B}$ for some finite set $B$ are called \emph{patterns}. 
Usually configurations will be denoted by letters like $x, y$ and $z$ while patterns will be denoted by letters like $a, b$ and $c$. 
Given a configuration $x$, let $x_\mi:=x(\mi)$ and a pattern $a\in \A^B$ and $\mi\in B$, let $a_\mi:=a(\mi)$.

There is a natural action of $\Z^d$ on $\A^{\Z^d}$: For all $\mi \in \Z^d$ let 

$$\sigma^\mi:\A^{\Z^d}\longrightarrow \A^{\Z^d}\text{ given by }\left(\sigma^\mi(x)\right)_{\mj}:=x_{\mi+ \mj}$$ denote the \emph{shift-action}. A \emph{shift space} is a closed set of configurations $X\subset \A^{\Z^d}$ which is invariant under the shift-action, meaning, $\sigma^\mi(X)=X$ for all $\mi \in \Z^d$. Alternatively, it can also be defined using forbidden patterns: A set of configurations $X$ is a shift space if and only if there is a set of patterns $\F$ such that 
$$X=X_\F:=\left\{x\in \A^{\Z^d}~:~\text{ patterns from }\F\text{ do not appear in any shift of }x\right\}.$$ 
Look at \cite[Chapter 6]{LM} for the proof of the equivalence when $d=1$; the proof is similar in higher dimensions. In a similar fashion the shift map extends to patterns:
$$\sigma^\mi:\A^{F}\longrightarrow \A^{F-\mi}\text{ given by }\left(\sigma^\mi(a)\right)_{\mj}:=x_{\mi+ \mj}\text{ for }F\subset \Z^d\text{ and }\mj\in F-\mi$$
 Let $\m 0$ be the origin and $\{\m e_1^d, \m e_2^d, \ldots, \m e^d_d \}$ denote the standard generators of $\Z^d$. We will drop the superscript when it is obvious from the context. Given $a, b\in \A$ we denote by $\langle a,b \rangle^{i}\in \A^{\{\m 0, \m e_i\}}$ the pattern
$$\langle a,b\rangle ^i_{\m 0}:= a, \langle a,b \rangle ^i_{\m e_i}=b.$$

Let us look at a few examples:
\begin{enumerate}
\item
Let $\A=\{0,1\}$ and $\F=\{\langle1,1\rangle ^i~:~1\leq i \leq d\}.$ Then 
$$X_\F=\{x\in \{0,1\}^{\Z^d}~:~\text{ no two appearances of }1\text{ in }x\text{ are adjacent}\}.$$ 
This is called the \emph{hard square shift.}
\item 
Let $\A=\{1,2,\ldots, n\}$ and $\F=\{\langle j,j\rangle ^i~:~1\leq i \leq d, 1\leq j\leq n\}.$
Then 
$$X_\F=\{x\in \{1,2, \ldots, n\}^{\Z^d}~:~\text{ adjacent symbols}\text{ in }x\text{ are distinct}\}.$$ 
This is called the \emph{$n$-coloured chessboard}.
\item
Let $d=1$, $\A=\{0,1\}$ and $\F=\{10^{2i-1}1~:~i\in \Z\}$. Then 
$$X_\F=\{x\in \{0,1\}^\Z~:~\text{ the separation between successive }1\text{'s is even}\}.$$
This is called the \emph{even shift}. 

\end{enumerate}•

Note that in the hard square shift the forbidden list $\F$ consists of $d$ elements while in the even shift the forbidden list $\F$ consists of infinitely many elements. It can be in fact proven that $\F$ cannot be chosen finite for the even shift. 

A shift space $X$ is called a \emph{shift of finite type (SFT)} if there exists a finite set of forbidden patterns $\F$ such that $X=X_\F$. Thus the hard square shift is an SFT while the even shift is not an SFT. Further if $\F$ can be chosen to be a set of patterns on edges and vertices of $\Z^d$ then $X$ is called a \emph{nearest neighbour shift of finite type}. Any SFT can be ``recoded'' into a nearest neighbour SFT: Given shift spaces $X$ and $Y$, a continuous map $f: X\longrightarrow Y$ which commutes with the shift-action, that is, $f\circ \sigma^\mi=\sigma^\mi\circ f$ is called a \emph{sliding block code}. A \emph{factor map} is a sliding block code which is surjective while a \emph{conjugacy} is a sliding block code which is bijective. The inverse of a conjugacy is also a conjugacy; thus conjugacies determine an equivalence relation. Any shift space conjugate to an SFT is also an SFT. Further given an SFT $X$, a simple construction gives us a nearest neighbour SFT, $Y$ which is conjugate to $X$ \cite{schimdt_fund_cocycle_98}. 

A \emph{periodic configuration} is a configuration $x\in \A^{\Z^d}$ such that there exists some $n \in \N$ such that $\sigma^{n\m e_i}(x)=x$ for all $1\leq i\leq d$. Some fundamental properties of nearest neighbour SFTs are undecidable for $d\geq 2$; for instance there is no algorithm to decide, given a finite set $\F$ whether $X_\F$ is non-empty \cite{bergerundecidable, Robinson1971}. Let us review a few salient features of the proof: Fix $d\geq 2$. Given a Turing machine $T$ there is a finite alphabet $\A_T$ and a finite forbidden list $\F_T$ such that $X^d_{\F_T}$ is non-empty if and only if $T$ does not halt starting on the empty input. Since the halting problem for Turing machines is undecidable, the non-emptiness problem for SFTs (and hence nearest neighbour SFTs) is also undecidable. Further $X^d_{\F_T}$ has no periodic configurations; this shall be useful later.

All the graphs $\H$ in this paper are undirected, without multiple edges and have no isolated vertices.

$X\subset \A^{\Z^{d}}$ is called a \emph{hom-shift} if there exists a finite undirected graph $\H$ such that $X= Hom(\Z^d,\H)$. Alternatively, these are exactly the nearest neighbour SFTs which are symmetric and isotropic, meaning nearest neighbour SFTs which are invariant under the automorphism group of $\Z^d$ (as a graph). These correspond to vertex shifts in $d=1$ defined by an undirected graph \cite[Chapter 2]{LM}.

For an undirected graph $\H$ (finite or not) we denote 
$$X^d_\H:=Hom(\Z^d, \H).$$
Clearly $X^d_\H$ is non-empty if and only if $\H$ is non-empty. Let $K_n$ denote the complete graph on $n$ vertices $\{1, 2, 3, \ldots, n\}$. Then $X_{K_n}^d$ is the $n$-coloured chessboard. If $\H$ is the graph given by Figure \ref{figure:hardsquaregraph} then $X^d_\H$ is the hard square shift.

\begin{figure}
\includegraphics[angle=0,
width=.15\textwidth]{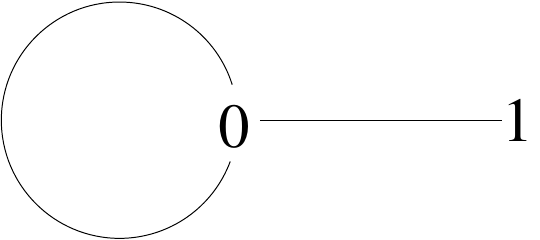}
\caption{Graph for the hard square shift\label{figure:hardsquaregraph}}
\end{figure}

We shall frequently use the cartesian product on graphs: Given graphs $\H_1=(\V_1, \E1_1)$ and  $\H_2=(\V_2,\E1_2)$, $\H_1\square \H_2$ is the graph with vertex set $\V_1\times \V_2$ where $(v_1, v_2)\sim_{\H_1\square \H_2}(w_1, w_2)$ if and only if $v_1=w_1$ and $v_2\sim_\H w_2$ or $v_1\sim_\H w_1$ and $v_2=w_2$. By $\square^r_{j=1} \H_j$ we mean the graph $\H_1\square \H_2\square\ldots\square\H_r$.

For a shift space $X\subset \A^{\Z^d}$, the \emph{language} for $X$ is given by
$$\L(X):=\{a\in \A^B~:~ N\subset \Z^d \text{ is finite and there exists }x\in X \text{ such that }x|_B=a\}.$$
These are called the set of \emph{globally-allowed patterns} in $X$. On the other hand, if the shift space $X$ is given by a forbidden list $\F$, then a pattern $a$ is called \emph{locally-allowed} if no element of $\F$ appears in the shifts of $a$. For shifts of finite type, it is not decidable whether a locally-allowed pattern is globally-allowed \cite{Robinson1971}. For hom-shifts, it is in fact decidable; this follows from Proposition \ref{Proposition: periodic points in Hom shifts}.

A \emph{shape} is a finite subset of $\Z^d$. For a shape $A\subset \Z^d$ we write $\L_A(X):=\L(X)\cap \A^A$. We will often denote an element $a\in \A^A$ by $\langle a\rangle_A$ instead to emphasise the domain of the pattern. By a \emph{rectangular shape} $A\subset \Z^d$ we mean that $A=\square^d_{j=1} I_j$ for some finite intervals $I_j\subset \Z$. A \emph{rectangular pattern} in $X$ is a pattern in $\L_A(X)$ for some rectangular shape $A$. The following proposition implies that periodic configurations are dense in hom-shifts.

\begin{prop}[Extension of (Possibly Infinite) Rectangular Patterns]\label{Proposition: periodic points in Hom shifts}
Let $\H$ be an undirected graph and $A=\square_{t=1}^d I_t$ where $I_t$'s are intervals in $\Z$. Then for all homomorphisms $a\in Hom(A, \H)$ there exists a configuration $x\in X^d_\H$ such that $x|_A=a$. If $A$ is a finite set then $x$ can be chosen to be periodic.
\end{prop}
Here is the idea: Let us first observe this for a finite $A$. If any of the side-lengths of $A$ is one then we extend it to a pattern $\t a$ on a bigger rectangular shape by `stacking shifts' of the pattern $a$. Then we reflect the pattern obtained about its faces to obtain a pattern $b$ on a still bigger rectangular shape and finally tile $\Z^d$ by this new pattern to obtain a periodic configuration. Some of the details are provided in Part (2) of the proof of \cite[Lemma 8.2]{MR3645513}. Although the proof there is for the case when $\H$ is a tree, it carries forward without any change to our context.

Now if $A$ is an (infinite) rectangular shape then by compactness of shift spaces and a standard limiting argument (taking a sequence of rectangular patterns which approximate the given pattern and considering the corresponding sequence of configurations extending them), the result for finite rectangular patterns implies the proposition.
\\

In the following, by a given nearest neighbour SFT $X$, we mean a given finite list of patterns $\F$ on edges and vertices of $\Z^d$ such that $X=X_\F$.

\begin{prop}\label{Proposition: decidability to periodic dense conjugacy}
Fix $d\geq 2$. Let $\mathcal C$ be a set of SFTs for which periodic points are dense for all $X\in \mathcal C$. It is undecidable whether an SFT is conjugate to some $X\in\mathcal C$. 
\end{prop}
\begin{proof}
Let $X\in \mathcal C$. Recall the properties of the SFT, $X_{\F_T}$ which was constructed given a Turing machine $T$. We can assume (possibly after a change in alphabet for $X$) that the underlying alphabets for $X$ and $X_{\F_T}$ are disjoint for all Turing machines $T$. Then $X\cup X_{\F_T}$ is a nearest neighbour SFT for every Turing machine $T$; since $X_{\F_T}$'s do not have periodic points, periodic points are dense in $X\cup X_{\F_T}$ if and only if $X_{\F_T}$ is empty. 

We claim that this implies $X\cup X_{\F_T}$ is conjugate to a member of $\mathcal C$ if and only if $X_{\F_T}$ is empty. Clearly, if $X_{\F_T}$ is empty then $X\cup X_{\F_T}\in \mathcal C$. Now suppose $X_{\F_T}$ is not empty. Since it does not have periodic points, periodic points are not dense in $X\cup X_{\F_T}$ and hence it cannot be conjugate to a member of $\mathcal C$.

Thus it is undecidable whether $X\cup X_{\F_T}$ is conjugate to an element of $\mathcal C$ proving, more generally, that it is undecidable whether a nearest neighbour SFT is conjugate to an element of $\mathcal C$. 
\end{proof}

\begin{corollary}\label{corollary: decidability conjugate to hom-shift}
It is undecidable whether a shift space $X$ is conjugate to a hom-shift for $d\geq 2$. 
\end{corollary}

This follows immediately from Propositions \ref{Proposition: periodic points in Hom shifts} and \ref{Proposition: decidability to periodic dense conjugacy}.

\section{Some Mixing Conditions for Hom-Shifts} \label{Section: Some Mixing Conditions for Hom-Shifts}
In this section we introduce some topological mixing conditions for shift spaces in $d\geq 2$. This introduction will be far from comprehensive; for more background consider \cite{boyle2010multidimensional}.

Given $A, B\subset \Z^d$ let
$$d_{\infty}(A, B):=\min_{\m i \in A, \m j \in B} {|\m i -\m j|_\infty}\text{ where }|\cdot|_\infty\text{ is the }l_{\infty}\text{ norm on }\R^d.
$$ 

A shift space $X$ is \emph{topologically mixing} or just \emph{mixing} if for all $\langle a \rangle_A, \langle b \rangle_B\in \L(X)$ there exists $n\in \N$ such that for all $\mi\in \Z^d$, $|\mi|_\infty\geq n$ there is $x\in X$ satisfying $x|_A=a$ and $\sigma^{\mi}(x)|_B=b$. A shift space $X$ is \emph{transitive} if for all $\langle a\rangle_A, \langle b\rangle _B\in \L(X)$ there exists $x\in X$ and $\m i\in \Z^d$ such that $x|_A=a$ and $\sigma^{\mi}(x)\Big\vert_B=b$. 

In this section we shall prove the following result:
\begin{prop}\label{prop: transitivity hom-shifts}
Let $d\geq 2$ and $\H$ be a finite undirected graph. Then $X^d_\H$ is transitive if and only if $\H$ is connected. Further it is mixing if and only if $\H$ is connected and not bipartite.
\end{prop}

Before we proceed with the proof, we shall consider a few more standard mixing conditions. A stronger mixing property which is also the main theme of this paper is the \emph{block-gluing property}: A shift space $X$ is said to be \emph{block-gluing} if there exists an $n\in \N$ such that for all rectangular patterns $\l a \r_A, \l b \r_B\in \L(X)$ satisfying $d_{\infty}(A,B)\geq n$ there exists $x\in X$ such that $x|_A=a$ and $x|_B=b$. A still stronger mixing condition is the following: A shift space $X$ is called \emph{strongly irreducible} (SI) if there exists $n \in \N$  such that for all $\l a \r_A, \l b \r_B\in \L(X)$ satisfying $d_{\infty}(A,B)\geq n$ there exists $x\in X$ such that $x|_A=a$ and $x|_B=b$.

The hard square shift $X$ is SI for $n=2$: Given shapes $A, B$ such that $d_{\infty}(A, B)\geq 2$ and $a\in \L_A(X), b \in \L_B(X)$, $x\in X$ given by 
$$x_\mi:=\begin{cases}
a_\mi&\text{ if }\mi \in A\\
b_\mi&\text{ if }\mi \in B\\
0&\text{ otherwise}
\end{cases}$$
satisfies $x|_A=a$ and $x|_B=b$. We will give a large class of examples in this paper of hom-shifts which are block-gluing and of hom-shifts which are mixing but not block-gluing. (Theorem \ref{theorem:fcfreecharacterisation}) We will also give an example of an hom-shift which is (phased) block-gluing but not (phased) SI in Subsection \ref{subsection:Why is Four-Cycle Hom-Free Condition Necessary}; the phased properties are introduced in Section \ref{section: The Block-Gluing Property for Hom-Shifts}.

\begin{prop}\label{prop: undecidability of transitivity/block-gluing}
Let $d\geq 2$. It is undecidable whether an SFT is transitive/mixing/block-gluing/SI.
\end{prop}

The proof is very similar to the proof of Proposition \ref{Proposition: decidability to periodic dense conjugacy}. Let $X$ be the hard square shift and consider for every Turing machine $T$ the SFT $X_{\F_T}$ (with alphabet disjoint from $\{0,1\}$); it is undecidable whether $X_{\F_T}$ is empty. Further $X\cup X_{\F_T}$ is transitive/mixing/block-gluing/SI if and only if $X_{\F_T}$ is empty; thus the proposition follows.
\\

Now let us return to Proposition \ref{prop: transitivity hom-shifts}. Suppose $\H$ is not connected. Let $\H=\H_1\cup \H_2$ where $\H_1$ and $\H_2$ are disjoint. Then $X^d_\H=X^d_{\H_1}\cup X^d_{\H_2}$ where $X^d_{\H_1}$ and $X^d_{\H_2}$ are non-empty shift spaces over disjoint alphabets proving that $X^d_\H$ is not transitive. Also if $\H$ is bipartite then $X^d_{\H}$ is not mixing since for a given $x\in X^d_{\H}$ and all even vertices $\mi \in \Z^d$, $x_{\mi}$ belong to the same partite class.

To prove the other direction we  will use some auxiliary constructions; the idea used for the proof of this proposition will be useful later as well.

A \emph{walk} $p$ in a graph $\H$ is a (finite, infinite or bi-infinite) sequence of vertices $\{p_i\}$ in $\H$ satisfying $p_{i}\sim_\H p_{i+1}$ for all $i$. A walk of \emph{length} $k$ is a finite walk $p=(p_0, p_1, \ldots, p_{k})$; let $|p|$ denote the length of $p$. Denote by $[i,j]$ the induced subgraph of $\Z$ on $\{i, i+1, \ldots , j\}$. For every $n \in \Z^+$ and $d\geq 2$ let 
$$B^{d-1}_n:=\square^{d-1}_{j=1} [-n,n],$$ that is, the $l^\infty$ ball of radius $n$ in $\Z^{d-1}$. Consider the graph $\H_{n,walk}^d:=(Hom(B^{d-1}_n, \H), \E1_{n,walk}^d)$ where
$$\E1^d_{n,walk}:=\{(x,y)~:~x_{\m i}\sim_\H y_{\m i}\text{ for all }\m i\in B^{d-1}_n\}.$$
As with homotopies in algebraic topology, there is a walk from $p$ to $q$ in $\H_{n,walk}^d$ of length $k$ if and only if there is a graph homomorphism $a: B^{d-1}_n\square [0,k]\longrightarrow \H$ such that $a_{\mi, 0}=p_\mi$ and $a_{\mi,k}=q_\mi$ for all $\mi\in B^{d-1}_n$. We will use this correspondence frequently throughout the paper.
Connectivity of the graph $\H_{n,walk}^d$ is related to the transitivity/mixing property via the following lemma:

\begin{lemma}\label{lemma: transitivity n-walk}
Let $d\geq 2$ and $\H$ be a finite undirected graph. If $\H_{n,walk}^d$ is connected for all $n\in \Z^+$ then $X^d_\H$ is transitive. Further if $\H_{n,walk}^d$ is connected and not bipartite for all $n\in \Z^+$ then $X^d_\H$ is mixing.
\end{lemma}

\begin{proof}
Let $A, B\subset \Z^d$ be finite sets and $\langle a\rangle_A, \langle b\rangle_B\in \L(X^d_\H)$ be given and suppose $\H^d_{n,walk}$ is connected. We need to prove that there exists some $\m i \in \Z^d$ such that $x|_A=a$ and $(\sigma^\mi(x))|_{B}=b$. By shifting the patterns if necessary and extending them to $B^d_n$ for some large enough $n>1$ we can assume $A=B= B^d_n$. By the hypothesis we know that $\H_{n,walk}^d$ is connected so there is a walk of length $k$ for some $k \in \N$ from $a|_{B^{d-1}_n\square{\{n\}}}$  to $b|_{B^{d-1}_n\square{\{-n\}}}$; here the graphs $B^{d-1}_n\square{\{-n\}}$ and $B^{d-1}_n\square{\{n\}}$ are identified with $B^{d-1}_n$. As observed earlier, this gives us a homomorphism $c: B^{d-1}_n\square{[n,n+k]}\longrightarrow \H$ such that $c_{\mi, n}=a_{\mi, n}$ and $c_{\mi, n+k}=b_{\mi, -n}$. `Pasting together' the configurations $a$ and $b$ to $c$ we get a homomorphism $l:B^{d-1}_n\square[-n,3n+k] \longrightarrow \H$ with $l|_{B^d_n}=a$ and 
$$l|_{B^{d}_n+(2n+k)\m e_d}=(\sigma^{-(2n+k)\m e_d}(b)).$$
 By Proposition \ref{Proposition: periodic points in Hom shifts} we see that $X^d_\H$ is transitive. 

For mixing, assume that $\H^d_{n,walk}$ is connected and not bipartite.  As before, let $\langle a\rangle_{B^{d}_{n}}, \langle b\rangle_{B^d_n}\in \L(X^d_\H)$. Choose an integer $k$ such that for all $a', b'\in \H^d_{n,walk}$ there is a walk from $a'$ to $b'$ of length $r$ for all $r\geq k$. Let $\mi=(i_1, i_2, \ldots, i_d)$ such that $|\mi|_\infty\geq k+2n$; without the loss of generality assume that $i_d\geq k+2n$. Extend $a$ and $b$ periodically to get extensions $\t a, \t b$ on $\Z^{d-1}\square[-n,n]$. There is a walk in $\H^d_{n,walk}$ from $\t a|_{B^{d-1}_n\square\{n\}}$ to $(\sigma^{-\mi}(\t b))|_{B^{d-1}_n\square\{-n+i_d\}}$ of length $i_d-2n$; thus we get a homomorphism $l': B^{d-1}_n\square{[-n,n+i_d]}\longrightarrow \H$ such that
$$l'|_{B^d_n}=\t a|_{B^d_n}\text{ and }l'|_{B^{d-1}_n\square[-n+i_d,n+i_d]}=(\sigma^{-\mi}(\t b))|_{B^{d-1}_n\square[-n+i_d,n+i_d]}.$$
By periodically extending $l'$ we get a homomorphism $\t l:\Z^{d-1}\square[-n,n+i_d]\longrightarrow \H$ such that
$$\t l|_{\Z^{d-1}\square[-n,n]}=\t a\text{ and }(\sigma^{\mi}(\t l))|_{\Z^{d-1}\square[-n,n]}=\t b.$$
By Proposition \ref{Proposition: periodic points in Hom shifts} the proof is complete.
\end{proof}

\begin{proof}[Proof of Proposition \ref{prop: transitivity hom-shifts}]
Fix $d\geq 2$. We have already shown that if $\H$ is not connected then $X^d_\H$ is not transitive. Let $\H$ be a connected graph. By Lemma \ref{lemma: transitivity n-walk} we need to prove that the graph $\H_{n,walk}^d$ is connected for all $n \in \Z^+$. When $n=0$, then $B^d_n$ consists of a single vertex; the connectivity of $\H_{0,walk}^d$ is exactly the connectivity of the graph $\H$. Now fix $n\geq 1$. The argument will follow by induction on $d$.
\smallskip

\noindent\textbf{Base Case:}
 Let $p, q\in \H_{n,walk}^2$. Consider a walk $r$ (say of length $k$) in $\H$ from $p_{n}$ to $q_{-n}$. Let $s:[-n,3n+k]\longrightarrow \H$ be the walk `joining' $p$, $r$ and $q$; formally, let

$$s_i:=\begin{cases}
p_i&\text{ if }i\in [-n,n]\\
r_{i-n}&\text{ if }i\in[n,n+k]\\
q_{i-2n-k}&\text{ if }i\in[n+k,3n+k].
\end{cases}•$$

By `stacking together the shifts' of the pattern $s$ we get a walk in $\H_{n,walk}^2$ from $p$ to $q$; formally, let $p^i\in\H_{n,walk}^2$ be given by 
$p^i_{t}:=s_{i+t}$ for $t\in [-n,n]$ and $i\in [0, 2n+k]$. Then $p^0=p$, $p^{2n+k}=q$ and 
$$p^i_t=s_{i+t}\sim_\H s_{i+t+1}=p^{i+1}_t$$
proving that $p^i\sim_{\H^2_{n,walk}}p^{i+1}$.

\noindent\textbf{The induction step:}
Let's assume the conclusion for some $d\geq 2$. Let $p,q\in \H_{n,walk}^{d+1}$. By the induction hypothesis there exists a walk $r^0, r^1,\ldots, r^k$ in $\H_{n,walk}^d$ from $p|_{[-n,n]^{d-1}\square\{n\}}$ to $q|_{[-n,n]^{d-1}\square\{-n\}}$ for some $k$. Let $s:[-n,n]^{d-1}\square[-n,3n+k]\longrightarrow \H$ be a graph homomorphism obtained by `joining' $p$, $r^0,r^1, \ldots, r^k$ and $q$; formally let

$$s_{\m j,i}:=\begin{cases}
p_{\mj,i}&\text{ if } i\in [-n,n]\\
r^{i-n}_{\mj}&\text{ if }i \in [n,n+k]\\
q_{\mj,i-2n-k}&\text{ if }i\in[n+k,3n+k]
\end{cases}$$
for all $\mj\in [-n,n]^{d-1}$. As in the base case, by `stacking together the shifts' of the pattern $s$ we get a walk from $p$ to $q$ in $\H_{n, walk}^{d+1}$. This proves that $X^d_\H$ is transitive.

If $\H$ is bipartite with partite classes $V_1, V_2$ and $x\in X^d_\H$ then $x_{\m 0}\in V_1$ if and only if $x_{\mi}\in V_1$ for all even vertices $\mi\in \Z^{d}$; thus $X^d_{\H}$ isn't mixing. For the other direction assume that $\H$ is connected and not bipartite. By the first part of the proof the graph $\H^d_{n,walk}$ is connected. Further since $\H$ is not bipartite it has an odd cycle. Thus one obtains an odd cycle in $\H^d_{n, walk}$ for all $n$; hence it is also not bipartite. By Lemma \ref{lemma: transitivity n-walk}, the proof is complete.
\end{proof}

Observe that the proof of Proposition \ref{prop: transitivity hom-shifts} gives us a bound on the diameter in the graph metric of $\H^{d+1}_{n,walk}$ given the diameter of $\H^{d}_{n, walk}$. Specifically
\begin{equation}
diam(\H^{d+1}_{n,walk})\leq 2n+diam(\H^{d}_{n,walk})\label{equation: nwalk d to d+1}
\end{equation}
for all $d\geq 0$; here $\H^{0}_{n,walk}$ is interpreted as the graph $\H$. We will be interested in cases where $diam(\H^{d+1}_{n,walk})$ is uniformly bounded for all $n$.

The following corollary follows from arguments in the proofs of Lemma \ref{lemma: transitivity n-walk} and Proposition \ref{prop: transitivity hom-shifts}.

\begin{corollary}\label{corollary: transitivity connectivity blah blah}
Let $\H$ be a finite undirected graph. The following are equivalent:
\begin{enumerate}
\item
$\H$ is connected.
\item
$X^d_\H$ is transitive for some $d\in \N$.
\item
$X^d_\H$ is transitive for all $d\in \N$.
\item
$\H^d_{n,walk}$ is connected for all $n$ and $d$.
\item
$\H^d_{n,walk}$ is connected for some $n$ and $d$.
\end{enumerate}
\end{corollary}
Let $\H$ be a bipartite connected graph with partite classes $V_1, V_2$. Then $X^d_\H=X_1\cup X_2$ where
$$X_i:=\{x\in X^d_\H~:~x_{\m 0}\in V_i\}.$$
To prove that if $\H$ is connected and not bipartite then $X^d_\H$ is mixing, the only place we used the fact that the graph $\H$ is not bipartite is to conclude that $\H^d_{n,walk}$ is also not bipartite. If $\H$ is connected and bipartite then $\H^d_{n,walk}$ is also connected and bipartite; there exists $K\in \N$ such that for any $p, q\in \H^{d}_{n, walk}$ and $k>K$ there is a walk from $p$ to $q$ of length either $k$ or $k+1$. It follows that $X_1$ and $X_2$ are mixing SFTs for the $(2\Z)^d$ action. So we have the following proposition:

\begin{corollary}\label{corollary:biparite2zdactiontransitive}
If $\H$ is a bipartite connected graph then $X^d_\H$ is a disjoint union of two conjugate mixing SFTs with respect to the $(2\Z)^d$ action.
\end{corollary}

 This is reminiscent of the case for $d=1$, where if $X$ is an irreducible SFT of period $p$ then it can be written as  disjoint union of $p$ conjugate mixing SFTs with respect to the $p\Z$ action (\cite[Exercise 4.5.6]{LM}). We shall state similar conclusions in Corollary \ref{corollary:bipartite2zdactionblock} for some stronger mixing properties. We remark that the group $(2\Z)^d$ (which is of index $2^d$ in $\Z^d$) can be replaced by any subgroup contained in the same partite class as $\m 0$ in these results. However for the ease of notation and understanding, we will work with the group $(2\Z)^d$ instead.

\section{The Phased Block-Gluing and SI Property for Hom-Shifts}\label{section: The Block-Gluing Property for Hom-Shifts}

From here on the graph $\H$ is connected unless stated otherwise. The graph metric on $\H$ is denoted by $d_\H$. The block-gluing property is too restrictive: If $\H$ is bipartite then $X^d_\H$ is not even mixing. With this in view, we define the following:

A shift space $X$ is said to be \emph{phased block-gluing} if there exists an $n\in \N$ and a finite set $S\subset \Z^d$ such that for all rectangular patterns $\l a \r_A, \l b \r_B\in \L(X)$ satisfying $d_{\infty}(A, B)\geq n$ there exists $x\in X$ such that $x|_A=a$ and $\sigma^{\m i}(x)|_B=b$ for some $\m i \in S$. The set $S$ will be called a \emph{gluing set} of $X$ and $n$ will be called a \emph{gluing distance}. Observe that although the phased block-gluing property is defined for finite rectangular patterns $\l a\r_A, \l b\r_B$, it immediately applies (by using the compactness of shift spaces) to infinite rectangular patterns as well.

From here on fix $d\geq 2$ unless mentioned otherwise. We will now construct some auxiliary graphs which will be useful in the study of the phased block-gluing property. Let $\H_{walk}^d=(X^{d-1}_\H, \E1_{walk}^d)$ be the graph where
$$\E1^d_{walk}=\{(x,y)~:~x_{\m i}\sim_\H y_{\m i}\text{ for all }\m i\in \Z^{d-1}\}.$$

\noindent Given symbols $v,w$ we denote by $(v,w)^{\infty,d-1}\in \{v,w\}^{\Z^{d-1}}$ the checkerboard configuration given by
$$(v,w)^{\infty,d-1}_{\mi}:=\begin{cases}
v&\text{ if }\mi\text{ is in the same partite class as }\m 0\\
w&\text{ otherwise.}
\end{cases}$$

\noindent Similarly $v^{\infty,d-1}$ is the constant configuration given by 
$$v^{\infty,d-1}_\mi:=v \text{ for all }\mi\in \Z^{d-1}.$$

 Let us look at a few examples. 

\begin{enumerate}
\item \label{example: walk space 1}
If $\H$ is a graph with a single edge and vertices $v,w$ then $X^{d-1}_\H$ consists only of the two checkerboard patterns $(v,w)^{\infty,d-1}$ and $(w,v)^{\infty, d-1}$ which are connected to each other in $\H_{walk}^d$.
\item \label{example: walk space 2}
Let $\H$ be the graph in Figure \ref{figure:hardsquaregraph} (the graph for the hard square shift). Since $0,1 \sim_\H 0$, for all $x\in X^{d-1}_\H$, $x\sim_{\H^d_{walk}}0^{\infty,d-1}$. In general, if $\H$ is a graph with a vertex $\star$ such that $\star\sim_\H v$ for all $v\in \H$ (in other words, if the hom-shift $X^{d-1}_\H$ has a so-called safe symbol) then for all $x\in X^{d-1}_\H$, $x \sim_{\H^d_{walk}}\star^{\infty,d-1}$.

\end{enumerate}

The usual graph metric on $\H^d_{walk}$ is denoted by $d^{w}_\H$. Further we say that $d^{w}_\H(x,y):=\infty$ if there is no finite walk from $x$ to $y$. The \emph{diameter} of $\H^d_{walk}$ is denoted by
$$diam(\H^d_{walk}):=\sup_{x, y \in \H^d_{walk}}d^{w}_\H(x,y).$$

The diameter of the graph $\H^d_{walk}$ measures the maximum distance required to transition between two configurations in $X^{d-1}_\H$. Recall the graphs $\H^d_{n, walk}$. They may be thought to `approximate' the graph $\H^d_{walk}$; in fact it follows quite easily that
$$diam(\H^d_{walk})=\infty\text{ if and only if }\lim_{n\longrightarrow \infty}diam(\H^d_{n,walk})=\infty.$$
The proof is left to the reader. Look also at Subsection \ref{subsection: growth of radius}.

As mentioned previously with respect to the graphs $\H^d_{n,walk}$, there is a correspondence between walks $x=p^0, p^1, \ldots, p^k=y$ in $\H^d_{walk}$ from $x$ to $y$ of length $k$ and $\t x \in Hom\left(\Z^{d-1}\square[0,k],\H\right)$ satisfying $\t x_{\mi,0}=x_\mi$ and $\t x_{\mi, k}=y_\mi$. We will use this and similar correspondences throughout the paper.

While the graphs $\H^d_{n, walk}$ were useful in analysing the mixing and transitivity of the hom-shifts $X^d_{\H}$ (as in Proposition \ref{prop: transitivity hom-shifts}), the graph $\H^d_{walk}$ relates to the phased block-gluing property by the following proposition:

\begin{prop}\label{prop:fromshifttowalk}
Let $\H$ be a finite, undirected graph. Then
\begin{enumerate}
\item $X^d_\H$ is block-gluing if and only if there exists an $n\in \N$ such that for all $x, y\in X^{d-1}_\H$ there exists a walk of length $n$ in $\H^d_{walk}$ starting at $x$ and ending at $y$.\label{prop:fromshifttowalk item 1}
\item  $X^d_\H$ is phased block-gluing if and only if $diam(\H^d_{walk})<\infty$. \label{prop:fromshifttowalk item 2}
\item If $\H$ is bipartite and $X^d_\H$ is phased block-gluing then the gluing set can be chosen to be $\{0, \m e_i\}$ for all $1\leq i\leq d$.\label{prop:fromshifttowalk item 4}
\item If $\H$ is not bipartite and $X^d_\H$ is phased block-gluing then $X^d_\H$ is block-gluing.\label{prop:fromshifttowalk item 3}
\end{enumerate}
\end{prop}

\begin{proof}[Proof of Part (\ref{prop:fromshifttowalk item 1}) of Proposition \ref{prop:fromshifttowalk}]

Suppose that $X^d_\H$ is block-gluing with gluing distance $n$. Let $x, y \in X^{d-1}_\H$. We can identify them as elements of $Hom(\Z^{d-1}\square\{0\}, \H)$ and $Hom(\Z^{d-1}\square \{n\}, \H)$ respectively. By the block-gluing property there exists $z\in X^{d}_\H$ for which
$z|_{\Z^{d-1}\square \{0\}}=x$ and $z|_{\Z^{d-1}\square{\{n\}}}=y$. Equivalently we have found a walk of length $n$ in $\H^d_{walk}$ from $x$ to $y$.

Conversely suppose that for all $x, y\in X^{d-1}_\H$ there exists a walk of length $n$ starting at $x$ and ending at $y$. Since we can always lengthen such a walk by revisiting a configuration adjacent to $y$, it follows that for all $x, y\in X^{d-1}_\H$, $m\geq n$ there is a walk of length $m$ from $x$ to $y$.

We would like to prove that $X^{d}_\H$ is block-gluing with block-gluing distance $n$. Let $\langle a \rangle_A, \langle b\rangle_B$ be two rectangular patterns in $X^{d}_\H$ such that $d_{\infty}(A, B)=m$. Using the symmetry and isotropy in hom-shifts and translating the patterns (if necessary), by Proposition \ref{Proposition: periodic points in Hom shifts} we can assume that $A\subset \Z^{d-1}\square [-r,r]$ and $B\subset \Z^{d-1}\square[m+r, m+r+k]$ for some $r, k\in \N$. Consider $\t y\in Hom(\Z^{d-1}\square [-r,r], \H)$ and $\t z\in Hom(\Z^{d-1}\square[m+r, m+r+k], \H)$ such that $\t y|_A=a$ and $\t z|_B=b$. Then there exists a walk $p^0, p^1, \ldots, p^m$ from $\t y|_{\Z^{d-1}\square{\{r\}}}$ to $\t z|_{\Z^{d-1}\square\{m+r\}}$ in $\H^d_{walk}$. Hence we get a homomorphism $\t x\in Hom(\Z^{d-1}\square [-r,m+r+k], \H)$ such that $\t x|_{\Z^{d-1}\square[-r,r]}=\t y$ and $\t x|_{\Z^{d-1}\square[m+r, m+r+k]}=\t z$. By Proposition \ref{Proposition: periodic points in Hom shifts} there exists $x\in X^{d}_\H$ such that $x|_{A}=a$ and $x|_{B}=b$. 
\end{proof}

In the following proof by $|\cdot |_1$ we mean the $l_1$ metric on $\R^d$.

\begin{proof}[Proof of Part (\ref{prop:fromshifttowalk item 2}) of Proposition \ref{prop:fromshifttowalk}]
Suppose that $X^d_\H$ is phased block-gluing with gluing distance $n$ and gluing set $S$. Choose $m\geq n$ large enough such that $m>|\m i|_1$ for all $\mi \in S$. Let $x, y\in X^{d-1}_\H$ be given.  As before we identify $x$ and $y$ as configurations in $Hom(\Z^{d-1}\square\{0\})$ and $Hom(\Z^{d-1}\square\{m\})$ respectively. By the phased block-gluing property there exists $z\in X^d_\H$ such that $z|_{\Z^{d-1}\square\{0\}}=x$ and $\sigma^{\mi}(z)|_{\Z^{d-1}\square \{m\}}=y$ for some $\mi \in S$. Write $\mi=(\mi^f, i_d)$ where $\mi^f\in \Z^{d-1}$. Then  
$$z_{\mj, m+i_d}=y_{\mj-\m i^f}\text{ for all }\mj\in \Z^{d-1}.$$
 Thus we have obtained a walk from $x$ to $\sigma^{-\mi^f}(y)$ in $\H^d_{walk}$ of length $m+i_d$. By using the fact that $z'\sim_{\H^d_{walk}}\sigma^{\m e_j^{d-1}}(z')$ for all $1\leq j\leq {d-1}$ and $z'\in X^{d-1}_{\H}$ we get a walk from $\sigma^{-\mi^f}(y)$ to $y$ of length $|-\mi^f|_1$. Thus
$$diam(\H^d_{walk})\leq \max_{\mi \in S} (m + |\mi|_1).$$

Now let us prove the converse. Suppose $diam(\H^d_{walk})< n<\infty$. Let $1\leq j \leq d$, $S=\{\m 0, \m e^d_j\}$ and $\langle a\rangle_A, \langle b\rangle_B\in \L(X^d_\H)$ be rectangular patterns such that $d_{\infty}(A, B)=m\geq n+1$. We can assume that $A\subset \Z^{d-1}\square [-r,r]$ and $B\subset \Z^{d-1}\square[m+r, m+r+k]$ for some $r, k\in \N$. Consider $\t y\in Hom(\Z^{d-1}\square [-r,r], \H)$ and $\t z\in Hom(\Z^{d-1}\square[m+r, m+r+k], \H)$ such that $\t y|_A=a$ and $\t z|_B=b$. There is a walk of length either $m-1$ or $m$ from $\t y|_{ \Z^{d-1}\square\{r\}}$ to $\t z|_{\Z^{d-1}\square\{m+r\}}$ since there is always a walk of length $2$ from any vertex in $\H^d_{walk}$ to itself.

\begin{enumerate}[{Case }(1){: }]
\item\emph{A walk of length $m$ is found:} We get $\t x\in Hom(\Z^{d-1}\square[-r,m+r+k], \H)$ such that $\t x|_{\Z^{d-1}\square [-r,r]}= \t y$ and $\t x|_{\Z^{d-1}\square [m+r, m+r+k]}= \t z$. By Proposition \ref{Proposition: periodic points in Hom shifts} there exists $x\in X^d_\H$ such that $x|_{A}=a$ and $x|_B=b$.

\item\emph{A walk of length $m-1$ is found:} This is similar to the previous case; just replace the pattern $\t z$ by $\sigma^{-\m e_j^d}(\t z)$.

\end{enumerate}
\end{proof}
\begin{proof}[Proof of Part (\ref{prop:fromshifttowalk item 4}) of Proposition \ref{prop:fromshifttowalk}]Note that we have proved that the phased block-gluing property for $X^d_\H$ implies that  $diam(\H^d_{walk})<\infty$ and that $diam(\H^d_{walk})<\infty$ implies that $X^d_\H$ has the phased block-gluing property where the gluing set $S$ can be chosen to be $\{\m 0, \m e_i\}$ for $1\leq i\leq d$. Thus, if $X^d_\H$ is phased block-gluing then the gluing set $S$ can be chosen to be $\{\m 0,\m e_i\}$ for $1\leq i\leq d$. 
\end{proof}

\begin{proof}[Proof of Part (\ref{prop:fromshifttowalk item 3}) of Proposition \ref{prop:fromshifttowalk}]
Suppose $\H$ is a finite, undirected graph which is not bipartite and $X^d_\H$ is phased block-gluing. If $\H$ is a single vertex with a self-loop then $\H^d_{walk}$ is single configuration with a self-loop as well; there is nothing to prove. If $\H$ is not a single vertex with a self-loop then since $\H$ is not bipartite there exist cycles of even and odd length in $\H$ and (hence) in $\H^d_{walk}$. Thus the graph $\H^d_{walk}$ is aperiodic.

Moreover since $X^d_\H$ is phased block-gluing, from Part (\ref{prop:fromshifttowalk item 2}) of this Proposition we know that $\H^d_{walk}$ has finite diameter. Since $\H^d_{walk}$ is aperiodic and has finite diameter, from standard arguments (look in \cite[Lemma 6.6.3]{PTEdurret}) one can prove that the adjacency matrix of the graph $\H^d_{walk}$ is primitive, meaning, there exists $m\in \N$ such that for every $x, y\in X^{d-1}_\H$ there exists a walk of length $m$ from $x$ to $y$ in $\H^d_{walk}$. By Part (\ref{prop:fromshifttowalk item 1}) the proof is complete.
\end{proof}

In exactly the same way, the phased SI property can also be defined: A shift space $X$ is said to be \emph{phased SI} if there exists an $n\in \N$ and a finite set $S\subset \Z^d$ such that for all patterns $\l a \r_A, \l b \r_B\in \L(X)$ satisfying $d_{\infty}(A, B)\geq n$ there exists $x\in X$ such that $x|_A=a$ and $\sigma^{\m i}(x)|_B=b$ for some $\m i \in S$. $S$ will be called an \emph{SI gluing set} of $X$ and $n$ will be called an \emph{SI gluing distance}.
\begin{prop}
\label{prop:Phaed SI}
Let $\H$ be a finite, undirected graph. Then
\begin{enumerate}
\item
If $\H$ is bipartite and $X^d_{\H}$ is phased SI, the SI gluing set can be chosen to be $\{\m 0,\m e_i \}$ for all $1\leq i\leq d$.
\item
If $\H$ is not bipartite and $X^d_{\H}$ is phased SI then it is SI.
\end{enumerate}
\end{prop}
Since the arguments for the proof of this proposition are similar to those in the proof of Proposition \ref{prop:fromshifttowalk}, we will not repeat them here. Roughly speaking, in Proposition \ref{prop:fromshifttowalk} we obtained the result by translating the question into one about walks on the auxiliary graphs $\H^d_{walk}$. For SI we can use the following simple equivalence instead: Given a set $A\subset \Z^d$ let 
$$\partial_r A\:=\{\mi \in \Z^d\setminus A~:~|\mi-\mj|_1\leq r\text{ for some }\mj\in A\}.$$
A nearest neighbour SFT $X$ is SI if and only if there is an $N\in \N$  such that for all $n\geq N$, finite $A\subset \Z^d$ and $\langle a\rangle_A, \langle b\rangle_{\partial_n A\setminus \partial_{n-1} A}\in \L(X)$, there exists $x\in X$ such that $x|_{A}=a$ and $x|_{\partial_n A\setminus \partial_{n-1} A}=b$.

As in Corollary \ref{corollary:biparite2zdactiontransitive} we can also conclude:
\begin{corollary}\label{corollary:bipartite2zdactionblock}
Let $\H$ be a bipartite finite undirected graph. If $X^d_\H$ is phased block-gluing/phased SI then $X^d_\H$ is a union of two disjoint conjugate SFTs with respect to the $(2\Z)^d$ action which are block-gluing/SI respectively.
\end{corollary}

This follows from the fact that for a phased block-gluing/phased SI hom-shift, the gluing set/SI gluing set can be chosen to be $\{0, \m e_i\}$ for all $1\leq i \leq d$. The proof is left to the reader.

We will need the following `monotonicity' result:

\begin{prop}\label{proposition: mononotinicity of dimensiosn}
Let $\H$ be a finite undirected graph and $d_1<d_2$. If $X^{d_1}_\H$ is not phased block-gluing/phased SI then $X^{d_2}_\H$ is not phased block-gluing/phased SI.
\end{prop}

Let us see this for the phased block-gluing property; the proof for the phased SI property uses similar ideas. Suppose $X^{d_1}_\H$ is not phased block-gluing. Fix $n \in \N$. By Proposition \ref{prop:fromshifttowalk} we know that $diam(\H^{d_1}_{walk})=\infty$. Thus there exists $x, y\in X^{d_1-1}_\H$ such that $d^w_\H(x,y)\geq n$. By Proposition \ref{Proposition: periodic points in Hom shifts} there exists $x^1,  y^1 \in X^{d_2-1}_\H$ such that $x^1_{(\mi, \m 0)}=x_\mi$ and $y^1_{(\mi, \m 0)}=y_\mi$ for all $\mi \in \Z^{d_1-1}$. Now given a walk (if it exists) $x^1, x^2, \ldots, x^k=y^1$ from $x^1$ to $y^1$ in $\H^{d_2}_{walk}$,
$$x^1|_{\Z^{d_1-1}\square \{\m 0\}}, x^2|_{\Z^{d_1-1}\square \{\m 0\}}, \ldots, x^k|_{\Z^{d_1-1}\square \{\m 0\}}$$
 is a walk in $\H^{d_1}_{walk}$ (up to identification of $\Z^{d_1-1}\square \{\m 0\}$ with $\Z^{d_1-1}$). Hence $d^w_\H(x^1, y^1)\geq n$. Since $n$ was arbitrary we have proven that $diam(\H^{d_2}_{walk})=\infty$ proving that $X^{d_2}_\H$ is not phased block-gluing.

We end this section with a few minor structural remarks. Let $C_n$ denote the $n$-cycle with vertices $\{0,1,2, \ldots, n-1\}$. The phased SI/phased block-gluing property for transitive hom-shifts is not stable under containment:  For instance we will prove that $X^{2}_{C_3}$ is not phased block-gluing in Theorem \ref{theorem:fcfreecharacterisation}. However $X^2_{Edge}$ and $X^2_{K_4}$ are both phased SI \cite{Raimundo2014} where $Edge$ is the induced subgraph on a pair of vertices in $C_3$ and $C_3$ is isomorphic to an induced subgraph of $K_4$. The mixing properties are however preserved under certain products:

The tensor product of graphs $\H_1=(\V_1, \E1_1)$ and $\H_2=(\V_2,\E1_2)$, denoted by $\H_1\times\H_2$ is the graph with vertex set $\V_1\times\V_2$ and $(v_1,v_2)\sim_{\H_1\times\H_2}(w_1,w_2)$ if $v_1\sim_{\H_1}w_1$ and $v_2\sim_{\H_2}w_2$.

\begin{prop}\label{proposition: graph products}
Let $\H_1$ and $\H_2$ be graphs such that $X^d_{\H_1}$ and $X^d_{\H_2}$ are phased SI/phased block-gluing. Let $\H$ be a connected component of $\H_1\times \H_2$. Then $X^d_{\H}$ is also phased SI/phased block-gluing.
\end{prop}
We understand the case of the cartesian product to a much lesser extent and might be of interest for future work.
\begin{proof} 

There are three separate cases to consider: neither $\H_1$ nor $\H_2$ is bipartite, exactly one of $\H_1$ and $\H_2$ are bipartite and finally both $\H_1$ and $\H_2$ are bipartite. The proofs for the three cases are similar given the following well-known observations: If $\H_1$ and $\H_2$ are connected graphs which are not bipartite then $\H_1\times \H_2$ is connected and bipartite. If exactly one of $\H_1$ and $\H_2$ is bipartite and both are connected then $\H_1\times\H_2$ is also bipartite and connected. If both $\H_1$ and $\H_2$ are bipartite and connected then $\H_1\times\H_2$ has two graph components, both are connected bipartite graphs.

Since these three cases are very similar we shall only prove the theorem for the case where both $\H_1$ and $\H_2$ are not bipartite. Let $X^d_{\H_1}$ and $X^d_{\H_2}$ be phased SI (and hence SI given Proposition \ref{prop:Phaed SI}). Let $(x^1, y^1) ,(x^2,y^2)\in X^d_{\H_1\times\H_2}$. Let $n$ be the maximum of the SI gluing distances for $X^d_{\H_1}$ and $X^d_{\H_2}$. Let $A, B\subset \Z^d$ such that they are separated by distance $n$. Then there exists $(x, y)\in X^d_{\H_1\times \H_2}$ such that $x|_{A}=x^1|_A$, $x|_B=x^2|_B$, $y|_A=y^1|_A$ and $y|_B=y^2|_B$. The proof for the block-gluing property follows the same idea; we need to restrict to rectangular shapes $A$ and $B$.
\end{proof}

Finally we observe that the lack of the block-gluing property is equivalent to the graph $\H^d_{walk}$ being disconnected:

\begin{prop}\label{prop: disconnection of graphs}
Let $\H$ be a finite undirected graph. Then $diam(\H^d_{walk})=\infty$ if and only if $\H^d_{walk}$ is disconnected.
\end{prop}

\begin{proof} We will prove the proposition in the case when $\H$ is not bipartite; the proof for the bipartite case is similar and left to the reader. Let $diam(\H^{d}_{walk})=\infty$. Then either $\H^{d}_{walk}$ is disconnected or for all $n\in \N$ there exist configurations $x^n, y^n\in X^{d-1}_\H$ such that $d^w_{\H}(x^n, y^n)\geq n$. By choosing a large enough subpattern from these configurations it follows that there exists $k_n\in \N$ and $a^n, b^n\in \H^d_{k_n, walk}$ such that the shortest walk from $a^n$ to $b^n$ is of length greater than or equal to $n$. Since $\H$ is not bipartite, by Proposition \ref{prop: transitivity hom-shifts}, the hom-shift $X^d_{\H}$ is mixing. Thus there exist $x,y\in X^{d-1}_{\H}$ such that there exists $\mi_n\in \Z^{d-1}$ satisfying 
$$\sigma^{\mi_n}(x)|_{B^{d-1}_{k_n}}=a^n,\sigma^{\mi_n}(y)|_{B^{d-1}_{k_n}}=b^n\text{ for all }n\in\N.$$
It follows that $d^w_{\H}(x,y)=\infty$ implying that $\H^d_{walk}$ is disconnected. 

For the other direction, if $\H^d_{walk}$ is disconnected then its diameter is infinite; this follows from the definition of the diameter.
\end{proof}

\section{Phased Mixing Properties for Four-Cycle Hom-Free Graphs}\label{section:Phased Mixing Properties for $C_4$-Hom-Free Graphs}

We say that an undirected graph $\H$ is a \emph{four-cycle hom-free graph} if for all graph homomorphisms $f:C_4\longrightarrow \H$ either $f(0)=f(2)$ or $f(1)=f(3)$. Let us begin by unravelling the definition.

\begin{prop}\label{prop: when H fchom-free}
An undirected graph $\H$ is four-cycle hom-free if and only if $C_4$ is not a subgraph of $\H$ and if $v\in \H$ has a self-loop then $w_1, w_2\sim_\H v$ and $w_1, w_2\neq v$ implies $w_1\not\sim_\H w_2$.
\end{prop}

\begin{proof} Let us see the forward direction; the arguments for the backward direction are similar in nature and left to the reader. Suppose $\H$ is four-cycle hom-free. Since there exists no graph homomorphism $f\in Hom(C_4, \H)$ which is an embedding, the graph $C_4$ is not a subgraph of $\H$. Now suppose the vertex $v\in \H$ has a self-loop, $w_1, w_2\sim_\H v$ and $w_1, w_2\neq v$. Consider the map $f': C_4\longrightarrow \H$ given by $f'(0)=f'(1):=v$, $f'(2):=w_1, f'(3):=w_2$; it is a graph homomorphism if and only if $w_1\sim_\H w_2$. But for the map $f'$, $f'(0)\neq f'(2)$ and $f'(1)\neq f'(3)$. Thus by the four-cycle hom-free property of $\H$ it follows that $f'$ is not a graph homomorphism from where it follows that $w_1\not\sim_\H w_2$.
\end{proof}

It follows from Proposition \ref{prop: when H fchom-free} that a graph $\H$ without self-loops is four-cycle hom-free if and only if it is a \emph{four-cycle free graph} in the sense of \cite{MR3645513}, that is, $C_4$ is not a subgraph of $\H$. It was observed in \cite{MR3645513} that a homomorphism from $\Z^d$ to $\H$ can be lifted to the universal cover $\H_{uni}$ (defined below). This includes graphs $\H$ which are trees and cycles $C_n$ for $n\neq 4$. A particular case is that of $n=3$; $X^d_{C_3}$ is the space of proper $3$-colourings of $\Z^d$.

This condition was studied in \cite{Marcinfourcyclefree2014} in the context of reconfiguration problems; we remark that the so-called fundamental groupoid in that paper is intimately related to the universal cover of $\H$. If $\H=C_3$ then the lifts correspond to the so called height functions (\cite{lieb}).

In addition it follows from Proposition \ref{prop: when H fchom-free} that the graph for the hard square shift (Figure \ref{figure:hardsquaregraph}) satisfies the hypothesis. For trees with loops, we refer to \cite{Raimundopavlov2016} (Proposition 8.1 and its corollaries) for related results.

In this section we describe a procedure for deciding the mixing conditions of $X^d_{\H}$ for a four-cycle-hom-free graph. For this we require a notion of folding in graphs: We say that a vertex $v$ \emph{folds} into $w$ if $N_{\H}(v)\subset N_\H(w)$. In this case $\H\setminus\{v\}$ is called a \emph{fold} of the graph $\H$. A graph is called \emph{stiff} if it does not have any non-trivial folds. Starting with a finite graph $\H$ we can obtain a stiff graph by a sequence of folds; stiff graphs thus obtained are the same up to graph isomorphism \cite[Theorem 4.4]{brightwell2000gibbs}. A graph $\H$ is called \emph{dismantlable} if there exists a sequence of graphs $\H=\H_1, \H_2, \ldots, \H_n$ such that $\H_{i+1}$ is a fold of the graph $\H_i$  for every $i$ and $\H_n$ is a vertex with or without self-loop. If $\H$ is a connected dismantlable graph which is not an isolated vertex then it follows that the stiff graph obtained by successive folds of $\H$ is a vertex with a self-loop. A graph $\H$ is called \emph{bipartite-dismantlable} if there exists a sequence of graphs $\H=\H_1, \H_2, \ldots, \H_n$ such that $\H_{i+1}$ is a fold of the graph $\H_i$  for every $i$ and $\H_n$ is either a single edge or a single vertex with a self-loop. Graph folding was introduced in \cite{Nowwinkler} to study cop-win graphs; later in \cite{brightwell2000gibbs} it was observed that folding preserves a lot of properties of the graphs. Since a fold of a graph $\H$ is bipartite if and only if $\H$ is bipartite it follows that if a graph $\H$ is bipartite-dismantlable, then it is dismantlable if and only if $\H$ is not bipartite.

The following proposition essentially follows from arguments similar to those in the proof of Theorem 4.1 in \cite{brightwell2000gibbs} and we omit them here:

\begin{prop}\label{Proposition: phased SI for dismantlable graphs}
Let $\H$ be a bipartite-dismantlable graph. Then $X^d_\H$ is phased SI. If $\H$ is bipartite-dismantlable and $X^d_\H$ is SI then $\H$ is dismantlable.
\end{prop}
We can now state the main result of this section.
\begin{thm}\label{theorem:fcfreecharacterisation}
Let $\H$ be a four-cycle hom-free graph. The following are equivalent:
\begin{enumerate}[(a)]
\item
 $X^d_\H$ is phased SI.\label{enum:phased SI}
\item
$X^d_\H$ is phased block-gluing.\label{enum:phased block-gluing}
\item
$\H$ is bipartite-dismantlable. \label{enum: bipartite dismantlable}
\end{enumerate}
\end{thm}

 The four-cycle hom-free condition is necessary for these equivalences; we will discuss this further after the proof of Theorem \ref{theorem:fcfreecharacterisation}.

Since phased SI is stronger than phased block-gluing, clearly (\ref{enum:phased SI}) implies (\ref{enum:phased block-gluing}) and by Proposition \ref{Proposition: phased SI for dismantlable graphs}, (\ref{enum: bipartite dismantlable}) implies (\ref{enum:phased SI}). To complete the proof of the theorem we need to prove (\ref{enum:phased block-gluing}) implies (\ref{enum: bipartite dismantlable}). For this we need to introduce the universal cover. For more details, look at \cite{MR3645513} and references within (mainly \cite{Angluin80,Stallingsgraph1983}).

A graph homomorphism $\phi: \H' \longrightarrow \H$ is called a \emph{graph covering} if it is surjective and for all $v\in \H$, the restricted map $\phi|_{N_{\H'}(v)}$ is bijective onto $N_\H(\phi(v))$; the induced map from $X^d_{\H'}$ to $X^d_{\H}$ is denoted by $\t\phi$. There is some subtlety here. Undirected graphs $\H$ can be viewed as $1$-CW-complexes where the vertices form $0$-cells and the edges form the $1$-cells of the complex. If $\H$ has no self-loops, then clearly the condition for a map $\phi:\H'\longrightarrow \H$ to be a graph covering implies that it is a topological covering as well. However a topological covering space of a graph $\H$ viewed as a $1$-CW-complex may be different from the covering graph of $\H$ when $\H$ has a self-loop. For instance, let $\H$ be a graph with a single vertex and a self-loop and $\H'$ be a graph with exactly one edge connecting two vertices; $\H'$ is a covering graph of $\H$ however $\H$ is homeomorphic to $S^1$ as a CW-complex and its only covering spaces are itself and $\R$; neither of these are homeomorphic to $\H'$. 

To avoid confusion, by a covering space of $\H$ we mean the usual topological covering space of $\H$ and by a covering graph of $\H$ we mean it in the sense as defined above; these two notions coincide if $\H$ has no self-loops.

A \emph{universal covering graph} of $\H$, denoted by $\H_{uni}$ is a covering graph of $\H$ which is a tree; this is unique up to graph isomorphism. Alternatively it can be defined as the connected covering graph $(\H_{uni}, \phi_{uni})$ satisfying the following (universal) property: Given a covering graph map $\phi: \H'\longrightarrow \H$ there exists a covering graph map $\phi':\H_{uni}\longrightarrow \H'$ such that $\phi\circ \phi'=\phi_{uni}$. There is an explicit construction of these graphs: A \emph{non-backtracking walk} in a graph $\H$ is a finite walk in which subsequent steps do not use the same edge, that is, walks $p_1, p_2, \ldots, p_n$ such that $(p_i, p_{i+1})\neq (p_{i+2}, p_{i+1})$. Fix a vertex $v\in \H$. $\H_{uni}$ is the graph where the vertex set is the set of non-backtracking walks in $\H$ starting at the vertex $v$ and two non-backtracking walks $p$ and $q$ are adjacent in the graph if one extends the other by a single step. Choosing a different starting vertex $v$ gives us a graph isomorphic to $\H_{uni}$. It is a tree and the covering graph map $\phi_{uni}: \H_{uni}\longrightarrow \H$ is given by 
$$\phi_{uni}(p):=\text{terminal vertex of }p.$$

Let us look at a few examples: Non-backtracking walks in a tree cannot visit the same vertex twice and there is a unique non-backtracking walk joining two distinct vertices. Hence the universal cover of a tree $\H$ is isomorphic to $\H$. The non-backtracking walks in the graph $C_n$ starting at $0$ are the finite prefixes of the periodic walks
$$0,1, 2, 3, \ldots, n-1, 0, 1, \ldots,\text{ and }0,n-1, n-2, \ldots, 1, 0, n-1, \ldots,$$
Thus the universal covering graph of $C_n$ is $\Z$ and the covering graph map is $\Mod n:\Z\longrightarrow C_n$. 

Another important class of examples are the \emph{barbell graphs} $Bar_n$ for $n>2$ with vertices $\{1,2,3,\ldots,n\}$ and edges $\{(1,1), (1,2),(2,3),\ldots,(n-1,n),(n,n)\}$ (Figure \ref{figure:4-barbell}). The non-backtracking walks on $Bar_n$ starting at $1$ are the finite prefixes of the periodic walks
$$(1,1,2,3,\ldots, n-1, n,n, n-1, n-2, \ldots,2,1,1,\ldots )\text{ and }(1,2,3,\ldots, n-1, n,n, n-1, n-2, \ldots,2,1,1,\ldots)$$ 
proving that $(Bar_n)_{uni}=\Z$.
Thus though the cycles $C_n$ and the barbells $Bar_n$ seem unrelated a priori, their universal covers are the same. By Proposition \ref{prop:four-cycle hom-free graph lift to universal covers} it will follow that the corresponding hom-shifts are related to each other. The fact that $Bar_n$ does not satisfy the block-gluing property has been essentially observed in \cite{Raimundopavlov2016}.
\begin{figure}
\includegraphics[angle=0,
width=.2\textwidth]{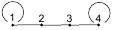}
\caption{Barbell graph for $n=4$\label{figure:4-barbell}}
\end{figure}

Let $\H$ be the graph for the hard square shift (given by Figure \ref{figure:hardsquaregraph}). The non-backtracking walks starting at the vertex $1$ are $(1)$, $(1,0)$, $(1,0,0)$ and $(1,0,0,1)$. Thus $\H_{uni}$  is isomorphic to the graph in Figure \ref{figure:liftofhardsquaregraph}.

\begin{figure}
\includegraphics[angle=0,
width=.1\textwidth]{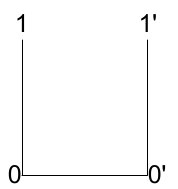}
\caption{Graph for the lift of the hard square shift\label{figure:liftofhardsquaregraph}}
\end{figure}

The universal covers of a graph are so-called normal covers \cite[Chapter 1]{hatchertopo}:

\begin{prop}\label{prop:universality of universal covers}
Let $\H$ be a finite undirected graph. For all $v',v''\in \H_{uni}$ satisfying $\phi_{uni}(v')=\phi_{uni}(v'')$ there is an automorphism $\psi$ of $\H_{uni}$ such that $\phi_{uni}\circ \psi=\phi_{uni}$ and $\psi(v')=v''$.
\end{prop}

A \emph{lift} of a configuration $x\in X^d_\H$ is a configuration $x'\in X^d_{\H_{uni}}$ such that $\t\phi_{uni}(x')=x$.

\begin{prop}\label{prop:four-cycle hom-free graph lift to universal covers}
Let $\H$ be a four-cycle hom-free graph. For all homomorphisms $x\in X^{d-1}_\H$, there exists a unique lift $x'\in X^{d-1}_{\H_{uni}}$ up to a choice of $x'_{\m 0}$. Further the induced map $\t \phi_{uni}$ is a graph covering map from $(\H_{uni})^{d}_{walk}$ to $\H^d_{walk}$.
\end{prop}

The proof of the first part of the proposition can be found in \cite[Proposition 6.2]{MR3645513}; the proof there is for four-cycle free graphs but it carries over for four-cycle hom-free graphs. For the second part, the same approach works with the added observation that $x\sim_{\H^{d}_{walk}}y$ if and only the configuration $z:\Z^{d-1}\square[0,1]\longrightarrow \H$ given by
$$z_{\mi, t}:=\begin{cases}
x_{\mi}\text{ if }t=0\\
y_\mi\text{ if }t=1
\end{cases}$$
is a graph homomorphism.

The proposition has immediate consequences for the phased block-gluing property:
\begin{corollary}\label{corollary: universal cover infinite no block-gluing}
Let $\H$ be a four-cycle hom-free graph. Then $diam(\H^d_{walk})<\infty$ if and only if $\H_{uni}$ is finite.
\end{corollary}

The proof shows that $diam(\H^d_{walk})<\infty$ for some $d\geq 2$ if and only if $diam(\H^d_{walk})<\infty$ for all $d\geq 2$; look also at Subsection \ref{subsection: Dependence on Dimension}.
\begin{proof}
Suppose $\H_{uni}$ is a finite graph (and hence a finite tree). By Proposition \ref{Proposition: phased SI for dismantlable graphs} and Part (\ref{prop:fromshifttowalk item 2}) of Proposition \ref{prop:fromshifttowalk} we get that $diam((\H_{uni})^d_{walk})<\infty$. Let $x, y \in X^{d-1}_\H$ and $x',y'$ be lifts of $x,y$ in $\H_{uni}$. There is a finite walk from $x'$ to $y'$ in $(\H_{uni})^d_{walk}$. By applying the induced map $\t \phi_{uni}$ to each step of the walk we get a walk of the same length from $x$ to $y$ in $\H^d_{walk}$. Thus $diam(\H^d_{walk})\leq diam((\H_{uni})^d_{walk})<\infty$.

Now suppose that $\H_{uni}$ is an infinite graph (and hence an infinite tree). By Proposition \ref{proposition: mononotinicity of dimensiosn} it is sufficient to prove that $diam(\H^2_{walk})=\infty$. Consider $x'\in X^{1}_{\H_{uni}}$ such that $x'|_{\N}$ does not visit the same vertex twice; since $\H_{uni}$ is a bounded degree infinite graph such an $x'$ exists. Let $x:=\t\phi_{uni}(x')$ and consider $y:=(v,w)^{\infty,1}$ for some edge $v\sim_\H w$. Suppose that there is a walk from $x$ to $(v,w)^{\infty, 1}$ in $\H^2_{walk}$. By Proposition \ref{prop:four-cycle hom-free graph lift to universal covers} it lifts to a unique walk from $x'$ to $y'=(v',w')^{\infty, 1}$ in $(\H_{uni})^2_{walk}$ for some $v', w'\in \H_{uni}$. 

Let $i_0\in \N$ be such that $d_{\H_{uni}}(x'_{i_0},v'):=\min_{i \in \N}d_{\H_{uni}}(x'_i,v')=:t$. Since $\H_{uni}$ is a tree it follows that $d_{\H_{uni}}(x'_{i_0},x'_i)=i-i_0$ for all $i\geq i_0$ and in fact
$$d_{\H_{uni}}(x'_{i},v')=i-i_0+t$$
 for all for all $i\geq i_0$. Therefore $$d^w_{\H_{uni}}(x',(v',w')^{\infty,1})=\infty$$
which leads to a contradiction and completes the proof.
\end{proof}

\begin{proof}[Proof of Theorem \ref{theorem:fcfreecharacterisation}]
Let $\H$ be a four-cycle hom-free graph.We are left to prove that (\ref{enum:phased block-gluing}) implies (\ref{enum: bipartite dismantlable}). By Corollary \ref{corollary: universal cover infinite no block-gluing} it is sufficient to prove that if $\H_{uni}$ is finite then $\H$ is bipartite-dismantlable. 

Now suppose that $\H_{uni}$ is a finite tree and hence is bipartite-dismantlable. We want to prove that $\H$ is bipartite-dismantlable. Suppose $v'$ folds into $w'$ in $\H_{uni}$, that is, $N_{\H_{uni}}(v')\subset N_{\H_{uni}}(w')$. Let $v:=\phi_{uni}(v')$ and $w:=\phi_{uni}(w')$. By Proposition \ref{prop:universality of universal covers} it follows that for all $v''\in \H_{uni}$ satisfying $\phi_{uni}(v'')=v$ there is an automorphism $\psi$ of $\H_{uni}$ for which $\phi_{uni}\circ \psi=\phi_{uni}$ and $\psi(v')=v''$. Thus for $w'':=\psi(w')$ we have that $\phi_{uni}(w'')=w$ and $v''$ folds into $w''$. Since $v'$ and $w'$ have common neighbours and $\phi_{uni}$ is a covering map it follows that $v\neq w$; in fact that $v$ folds into $w$. By folding all $v''$ which satisfy $\phi_{uni}(v'')=v$ we get $(\H\setminus\{v\})_{uni}$. The proof can be completed by induction on $|\H|$.
\end{proof}

\subsection{Why is the Four-Cycle Hom-Free Condition Necessary?}\label{subsection:Why is Four-Cycle Hom-Free Condition Necessary}\hfill\\

Some of the implications of Theorem \ref{theorem:fcfreecharacterisation} fail without the four-cycle hom-free assumption. We know that (\ref{enum:phased SI}) implies (\ref{enum:phased block-gluing}) for all shift spaces and by Proposition \ref{Proposition: phased SI for dismantlable graphs}, (\ref{enum: bipartite dismantlable}) implies (\ref{enum:phased SI}). Let us see why the other implications do not hold:

\begin{enumerate}
\item
\emph{(\ref{enum:phased SI})/(\ref{enum:phased block-gluing}) does not imply (\ref{enum: bipartite dismantlable}):} Here we see why the phased SI property in hom-shifts does not imply that the corresponding graph is bipartite-dismantlable. Let $K_n$ denote the complete graph with $n$ vertices, $1, 2, \ldots, n$. It is mentioned in \cite{Raimundo2014} that $X^d_{K_n}$ is SI for $n\geq 2d+1$; note that there is no folding possible in $K_n$ and hence it is not bipartite-dismantlable (except for $n=2$). Yet $X^d_{K_n}$ is block-gluing for $n\geq 4$ and $d\in \N$; this is proved in the following proposition. The argument given here is by Ronnie Pavlov; similar arguments appear in Section 4.4 of \cite{schmidt_cohomology_SFT_1995}.

A vertex in $\Z^{d-1}$ is called \emph{even} if it is in the same partite class as $\m 0$ and \emph{odd} otherwise.
\begin{prop} \label{Prop: 4colouringfinitediameter}
For $n \geq 4$, $diam((K_n)^d_{walk})\leq 4$.
\end{prop}
By Proposition \ref{prop:fromshifttowalk} this implies that $X^d_{K_n}$ is block-gluing for $n\geq 4$.
\begin{proof}
Let $x\in X^{d-1}_{K_n}$. Let $y \in X^{d-1}_{K_n}$ be a homomorphism given by 
$$y_{\m i}=\begin{cases}
1\text { if }\m i \text{ is even and }x_{\m i}\neq 1\\
2\text { if }\m i \text{ is even and }x_{\m i}=1\\
3\text{ if }\m i \text{ is odd and }x_{\m i}\neq 3\\
4\text{ if }\m i \text{ is odd and }x_{\m i}=3.
\end{cases}•$$
Clearly $x\sim_{\H^d_{walk}}y$ and $y\sim_{\H^d_{walk}} (3,1)^{\infty, d-1}$ (the checkerboard pattern in $3$ and $1$ which is 3 at $\m 0$).
Hence $d^w_{K_n}(x,(3,1)^{\infty, d-1})\leq 2$. Hence $diam((K_n)^d_{walk})\leq 4$.
\end{proof}
\item\label{item: phased block-gluing not phased SI}
\emph{(\ref{enum:phased block-gluing}) does not imply (\ref{enum:phased SI}): } Here we show the existence of a hom-shift which is phased block-gluing but not phased SI. It was mentioned to the authors by Raimundo Brice\~no \cite{Rai4cb3d} that $X^3_{K_4}$ is not phased SI (while by Proposition \ref{Prop: 4colouringfinitediameter} it is phased block-gluing). Here we shall give another example; this will be an instance of a large class of hom-shifts with the phased block-gluing property (Subsection \ref{Subsection:G-gluing property}). Let $\H$ be the graph given by Figure \ref{figure:SInotlbock}. We will prove that $X^d_{\H}$ is phased block-gluing for all $d\geq 2$ but not phased SI even for $d=2$.
\begin{figure}
\includegraphics[angle=0,
width=.5\textwidth]{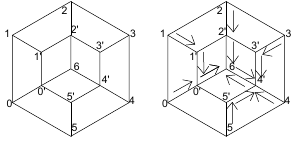}
\caption{On the left: Graph $\H$ for a hom-shift which is phased block-gluing but not phased SI. On the right: A graph homomorphism $f: \H \longrightarrow \H$ such that $f(v)\sim_\H v$ and $f^3(\H)$ is a single edge. \label{figure:SInotlbock}}
\end{figure}
Let us first observe why is $X^2_{\H}$ not phased SI. Fix $n\in \N$ and let $L$ be the shape given by
$$L:=\{(i,0), (n,i)~:~0\leq i\leq n\}.$$

Let $x\in X^2_\H$ be given by
$$x_{(j,k)}:=j+k\!\! \Mod 6.$$
Observe that for all $i\in \Z$, $i+1\!\! \Mod 6$ is the unique vertex in $\H$ adjacent to both $i\!\! \Mod 6, i+2\!\! \Mod 6$. It follows that $x_{(j+1,k)}$ is the unique vertex adjacent to $x_{(j,k)}$ and $x_{(j+1,k+1)}$ for all $(j,k)\in \Z^2$ which implies that if $y\in X^2_\H$ is a configuration such that $x|_{L}=y|_{L}$ then $x|_{[0,n]\square[0,n]}=y|_{[0,n]\square[0,n]}$. Thus $X^2_\H$ is not phased SI.
 
Now we will prove that $X^d_{\H}$ is phased block-gluing for all $d\geq 2$. Consider the map $f:\H\longrightarrow \H$ given by Figure \ref{figure:SInotlbock} and $d\geq 2$: For all $v\in \H$, $f(v)$ is defined to be the head of the arrow starting at $v$. Observe that $f$ is a graph homomorphism such that $f(v)\sim_{\H}v$ for all $v\in \H$ and $f^3(\H)$ is the edge joining vertices $4'$ and $6$. Thus for all $x\in X^{d-1}_\H$,
$f\circ x\sim_{\H^d_{walk}} x$ and $f^3\circ x$ is either $(4',6)^{\infty, d-1}$ or $(6,4')^{\infty, d-1}$ proving
$$d^w_{\H}(x, (4',6)^{\infty, d-1})\leq 4 \text{ and hence }diam(\H^d_{walk})\leq 8.$$
\end{enumerate}

\section{Further Directions}\label{section: open question}

\subsection{Decidability of the Fixed Block-Gluing Distance}\label{subsection:Decidability of Block-Gluing Distance} \hfill\\

\noindent\textbf{Question: }Fix $n \in \N$ and $d\geq 2$. Is there an algorithm to decide whether $diam(\H^d_{walk})=n$ for undirected graphs $\H$?

Let us see how such an algorithm may be constructed for certain dimensions. Fix $n\in \N$ and a graph $\H$. Recall, as in Section \ref{Section: Some Mixing Conditions for Hom-Shifts} the graph $\H^2_{n,walk}$ for which the vertices are homomorphisms from $[-n,n]$ to $\H$; two such homomorphisms $x,y$ are adjacent if $x_i\sim_{\H}y_i$ for all $i$. Consider the $d-1$ dimensional hom-shift constructed using this graph: $X^{d-1}_{\H^2_{n,walk}}$. Since this makes the notation onerous we will denote these shift spaces by $X^{d-1}_{\H,n}$. Let 
$$X^{d-1}_{\H,TB}:=\{(x,y)\in X^{d-1}_\H\times X^{d-1}_\H~:~\text{there is a walk of even length from }x_{\m 0}\text{ to }y_{\m 0}\}.$$
Observe that if $\H$ is not bipartite then $X^{d-1}_{\H, TB}=X^{d-1}_{\H}\times X^{d-1}_\H$; if it is bipartite then we further require that $x_{\m 0}$ and $y_{\m 0}$ are in the same partite class. There is a natural map from $\pi^{d-1}_{\H,n}:X^{d-1}_{\H,n}\longrightarrow X^{d-1}_{\H, TB}$ given by 
$\pi^{d-1}_{\H,n}(z):=(x,y)$ where 
\begin{eqnarray*}
x_{\mi}&:=&z_{\mi}(n)\\
y_{\mi}&:=&z_{\mi}(-n).
\end{eqnarray*}
This construction is related with the phased block-gluing property via the following proposition:
\begin{prop}\label{prop: block-gluing iff sofic at distance n}
Let $\H$ be an undirected graph. Then $X^d_{\H}$ is phased block-gluing for some block-gluing distance $2n$ if and only if the map $\pi^{d-1}_{\H,n}$ is surjective.
\end{prop}
\begin{proof}
By the proof of Proposition \ref{prop:fromshifttowalk}, $X^d_{\H}$ is phased block-gluing for distance $2n$ if and only if for all $x,y\in X^{d-1}_\H$ there exists a walk either from $x$ to $y$ or from $x$ to $\sigma^{\m e_1} (y)$ of length $2n$; equivalently, for all $x,y\in X^{d-1}_\H$ there exists $z\in X^{d-1}_{\H,n}$ such that either $\pi^{d-1}_{\H,n}(z)=(x,y)$ or $\pi^{d-1}_{\H,n}(z)=(x,\sigma^{\m e_1}(y))$. Consider a pair $(x',y')\in X^{d-1}_{\H,TB}$. The distance between $x'$ and $y'$ is even. Hence for $z'\in X^{d-1}_{\H,n}$, $\pi^{d-1}_{\H,n}(z')\neq (x,\sigma^{\m e_1}(y))$. Thus there exists $z''\in X^{d-1}_{\H,n}$ such that $\pi^{d-1}_{\H,n}(z'')=(x,y)$ completing the proof.
\end{proof}
\begin{thm}\label{thm: decidability in 2 dimensions}
It is decidable whether a hom-shift in two dimensions is block-gluing for distance $n$. 
\end{thm}
Recall, a shift space is called a \emph{sofic shift} if it is the image of an SFT under a sliding block-code.
\begin{proof}We will verify this only in the case when $n$ is even; for odd $n$, the proof is similar. By Proposition \ref{prop: block-gluing iff sofic at distance n} it is equivalent to verify that $Image(\pi^{1}_{\H,\frac{n}{2}})=X^{1}_{\H, TB}$. Now $X^{1}_{\H, TB}$ is an SFT (and hence sofic) and $Image(\pi^{1}_{\H,\frac{n}{2}})$ is sofic; there are well known algorithms to decide whether two sofic shifts are the same (\cite[Theorem 3.4.13]{LM}). This proves that it is decidable whether a hom-shift in two dimension is block-gluing for block-gluing distance $n$.
\end{proof}

Since it is undecidable whether a higher dimensional SFT is non-empty it automatically follows that that it is undecidable whether two $d-1$ dimensional sofic shifts are equal for $d\geq 3$. However even for $d=2$ we do not know the answer to the following questions:
\\

\noindent\textbf{Question:} Fix $n\in \N$. Is it decidable whether the SI gluing distance for a hom-shift is less than or equal to $n$?
\\

\noindent\textbf{Question:} Is the phased block-gluing/phased SI property decidable for hom-shifts?

\subsection{The gluing property for general boards $\G$}\label{Subsection:G-gluing property} \hfill\\

Our construction of the graph $\H^d_{walk}$ was motivated by the study of the block-gluing property. The question whether $diam(\H^d_{walk})<\infty$ can be viewed as a certain `reconfiguration' problem. A natural extension of the question is the following: Let $\G$ be a connected undirected graph without self-loops. Consider the graph
$$\H^\G_{walk}:=(Hom(\G, \H), \E1^\G_{walk})\text{ where }\E1^\G_{walk}:=\{(x,y)~:~x_i\sim_{\H}y_i\text{ for all }i \in \G\}.$$
\textbf{Question: }For which graphs $\H$ is $diam(\H^\G_{walk})<\infty$ for all undirected graphs $\G$?

For a reconfiguration problem of a similar nature, a characterisation was given in \cite{brightwell2000gibbs}: We say that $Hom(\G,\H)$ satisfies the pivot property if for all $x, y\in Hom(\G, \H)$ which differ only at finitely many sites there exists a sequence $x=x^1, x^2, \ldots, x^n=y\in Hom(\G, \H)$ such that $x^i, x^{i+1}$ differ at most at one site. Brightwell and Winkler proved that the pivot property is satisfied by $Hom(\G, \H)$ for all graphs $\G$ if and only if $\H$ is dismantlable. We wonder if a characterisation of similar nature exists in our case as well. In the following we provide a large class of graphs $\H$ for which $diam(\H^\G_{walk})<\infty$ for all connected undirected graphs $\G$.

We say that $\H$ is \emph{collapsible} if there exists a graph homomorphism $f:\H\longrightarrow \H$ such that $f(v)\sim_{\H}v$ for all $v\in \H$ and there exists $n \in \N$ such that $f^n(\H)$ is either an edge or a vertex with a self-loop; $f$ is called a \emph{collapsing map}. If $\H$ is a collapsible graph,  $diam(\H^\G_{walk})<\infty$ for all graphs $\G$ ((\ref{item: phased block-gluing not phased SI}), Subsection \ref{subsection:Why is Four-Cycle Hom-Free Condition Necessary}).

While one may feel that the proof that $diam((K_n)^d_{walk})<\infty$ for all $n \geq 4$ in Proposition \ref{Prop: 4colouringfinitediameter} is of a very different nature from that for the collapsible graphs, it can be shown that they are intimately related. Consider the covering graph map $\phi:\H\longrightarrow K_4$ given by $\phi(v')=\phi(v'')=v$ for all $v\in [1,4]$ where $\H$ is given by Figure \ref{figure:cover of K4}. As in Proposition \ref{prop:four-cycle hom-free graph lift to universal covers}, it is easy to see that for all homomorphisms $x\in X^{d-1}_{K_4}$, there exists a unique lift $x'\in X^{d-1}_{\H}$ up to a choice of $x'_{\m 0}$. Further the induced map $\t \phi$ is a graph covering map from $(\H)^{d}_{walk}$ to $(K_4)^d_{walk}$. One can thereby conclude that $diam(\H^{d}_{walk})<\infty$ if and only if $diam((K_4)^d_{walk})<\infty$. But the map $f:\H\longrightarrow \H$ given by Figure \ref{figure:cover of K4} is a collapsing map proving that $diam((K_4)^d_{walk})<\infty$.

\begin{figure}
\includegraphics[angle=0,
width=.5\textwidth]{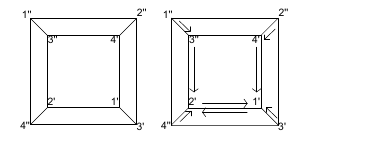}
\caption{A cover of $K_4$ on the left and its collapsing map on the right.\label{figure:cover of K4}}
\end{figure}

\subsection{The growth rate of the diameter of $\H^d_{n ,walk}$}\label{subsection: growth of radius} \hfill\\

We write that a sequence $a_n=\Theta(n)$ if there exists $c, C>0$ such that $cn\leq a_n\leq Cn$. 
\\

\noindent\textbf{Conjecture:} If $\H$ is a finite undirected graph $diam(\H^d_{walk})=\infty$ if and only if $diam(\H^d_{n,walk})=\Theta(n)$.

This was also conjectured by Ronnie Pavlov and Michael Schraudner who showed that this is true in several examples \cite{schraudnerpavlovblockgluingdist}. From Equation \ref{equation: nwalk d to d+1} we get a natural upper bound on the diameter:
$$diam(\H^{d}_{n,walk})\leq diam(\H)+2n(d-1).$$
If $\H$ is a four-cycle hom-free graph and $d\geq 2$ then it can be proved that $diam(\H^d_{walk})=\infty$ if and only if $diam(\H^{d}_{n,walk})=\Theta(n)$. We will prove the conjecture in the case when $\H$ is a four-cycle hom-free graph.

Suppose that $diam(\H^{d}_{n,walk})=\Theta(n)$. Since $diam(\H^{d}_{n,walk})$ is increasing in $n$ and converges to $diam(\H^d_{walk})$, it follows that $diam(\H^d_{walk})=\infty$. 

For the other direction assume that $diam(\H^d_{walk})=\infty$. Since $diam(\H^d_{n, walk})$ is increasing in $d$, it is sufficient to prove that $diam(\H^{d}_{n,walk})=\Theta(n)$ for $d=2$. By Corollary \ref{corollary: universal cover infinite no block-gluing}, $\H_{uni}$ is infinite. As in the proof of corollary let $x'\in X^1_{\H_{uni}}$ be such that $x'|_{\N}$ does not visit the same vertex twice and let $x:=\t\phi_{uni}(x')$. Then $d_{\H_{uni}}(x'_i, x'_j)=|i-j|$ for all $i, j \in \N$ implying that for all vertices $v'\in \H_{uni}$, there exists $i\in [0,2n]$ such that $d_{\H_{uni}}(x_i,v')\geq n$. This implies that the shortest walk in $\H^d_{n,walk}$ from $x|_{[0,2n]}$ to $(v,w)^{\infty,1}|_{[0,2n]}$ for all edges $v\sim_{\H}w $ is of length at least $n$. This proves that $diam(\H^{2}_{n,walk})=\Theta(n)$.

\subsection{Dependence on dimension} \label{subsection: Dependence on Dimension} \hfill\\

\noindent\textbf{Problem: } Construct a graph $\H$ for which $diam(\H^2_{walk})<\infty$ but $diam(\H^3_{walk})=\infty$.

In this paper we mention two large collection of graphs for which the $diam(\H^d_{walk})<\infty$ for all $d$: bipartite-dismantlable graphs (as in Section \ref{section:Phased Mixing Properties for $C_4$-Hom-Free Graphs}) and collapsible graphs (as in 
Subsection \ref{Subsection:G-gluing property}). However in all such examples, we find that $diam(\H^d_{walk})<\infty$ for all $d$. To find examples for the problem above, we would have to find a way to prove that $diam(\H^2_{walk})<\infty$ in a fundamentally different way.

By Proposition \ref{prop:fromshifttowalk}, the problem stated above is equivalent to the problem of finding a graph $\H$ for which $X^2_\H$ is block gluing but $X^3_\H$ is not block-gluing. We note that the answer to the analogue of this problem for SI is known: $X^2_{K_4}$ is SI \cite{Raimundo2014} but $X^3_{K_4}$ is not SI \cite{Rai4cb3d}.

\subsection{Block-gluing for periodic points}\label{subsection: block-gluing for periodic points}\hfill\\

\noindent\textbf{Problem: }Construct a graph $\H$ such that $d^w_\H(x,y)<\infty$ for all periodic points $x,y\in X^{d-1}_\H$ but $diam(\H^d_{walk})=\infty$.

If $diam(\H^d_{walk})=\infty$, by Proposition \ref{prop: disconnection of graphs} there exists some $x,y\in X^{d-1}_\H$ such that $d^w_\H(x,y)=\infty$ however it is not clear if $x,y$ can be chosen periodic. Such periodic points can be chosen if $\H$ is four-cycle free: By Corollary \ref{corollary: universal cover infinite no block-gluing}, $\H_{uni}$ is infinite and $\H$ is not a tree. Let $u_0,u_1, \ldots, u_{k-1}, u_k=u_0$ be a simple cycle in $\H$ for some $k>2$. Consider $x\in X^2_\H$ given by $x_i:=u_{i\Mod{k}}$ for all $i\in \Z$; $x$ is periodic. Let $x'\in X^1_{\H_{uni}}$ be any lift of $x$. Since $x_i\neq x_{i+2}$ for all $i \in \Z$ it follows that $x'_i\neq x'_{i+2}$ for all $i \in \Z$; because $\H_{uni}$ is a tree, this implies that $x'$ does not visit the same vertex twice. As in the proof of Corollary \ref{corollary: universal cover infinite no block-gluing} it follows that $d^w_\H(x,(v,w)^{\infty,1})=\infty$ for all $v\sim_\H w$.

\subsection{Measures of maximal entropy and Markov chains on $\H^2_{walk}$}\label{subsection: measures of maximal entropy}\hfill\\

 Given a shift space $X$ and $b\in \L_B(X)$ for some $B\subset \Z^2$, denote by 
$$[b]_B:=\{x\in X~:~x|_B=b\}\text{ the corresponding cylinder set}.$$ One of the motivations for studying the graph $\H^d_{walk}$ is also to understand the measures of maximal entropy on the space $X^d_{\H}$. Let us talk about the case $d=2$. There is a natural correspondence between stochastic processes $\nu$ on $\H^2_{walk}$ and probability measures $\mu$ on $X^2_\H$ given by
$$\nu(X^i_j=a_{i,j}\text{ for }(i,j)\in B):=\mu([a]_B)\text{ for }B\subset{\Z^2}\text{ finite and }a\in \L_B(X^2_\H).$$

For this subsection the necessary background for measures of maximal entropy can be gathered from \cite{Rue,burtonsteiffnonuniquesft} and for Markov chains from \cite[Chapter 6]{PTEdurret}. 
 Let $\H$ be a finite undirected graph and $\mu$ be an ergodic measure of maximal entropy for $X^2_\H$. Consider the Markov chain $\nu$ on $\H^2_{walk}$ obtained by the ``Markovisation'' of $\mu$ (look also at \cite[Chapter 1]{boweneqlecturenotes}):  Let $\pi$ be the probability measure on $X^1_\H$ given by marginalising $\mu$ to the vertical line $\{0\}\square \Z$. Consider the probability (also called Markov) kernel on $(\H^2_{walk}, \B)$, $\kappa: X^1_{walk}\times \B\longrightarrow [0,1]$ given by
$$\kappa(x,[y]_{-n,n}):=\mu(X_{(1,i)}=y_i\text{ for }i\in [-n,n]\vert X_{(0,i)}=x_i\text{ for }i\in\Z);$$
it is well-defined for $\pi$-almost every $x$.

Since $\mu$ is a shift-invariant probability measure it follows that $\pi$ is a stationary measure for the kernel $\kappa$. It can be proved that the measure $\t \mu$ on $X^2_\H$ corresponding to the Markov chain $\nu$ is also a measure of maximal entropy.
\\

\noindent\textbf{Conjecture: } Let $\H$ be a finite undirected graph and $\mu$ be an ergodic measure of maximal entropy on $X^2_\H$. Then the stochastic process on $\H^2_{walk}$ corresponding to $\mu$ is a Markov chain.

A study of random walks on the graph $(C_3)^2_{n,walk}$ can be found in \cite{norrisdiffuse2015}. 

\subsection{When is an SFT conjugate to a hom-shift}\hfill\\
\label{subsection: homshift sft conjugate}
\smallskip

\noindent\textbf{Question:} Let $d=1$. Is it decidable whether an SFT is conjugate to a hom-shift?

For $d\geq 2$ we have already observed in Corollary \ref{corollary: decidability conjugate to hom-shift} that it is undecidable whether an SFT is conjugate to a hom-shift.

\section{Acknowledgements}
The first author would like to thank Prof. Neeldhara Mishra for introducing him to \cite{Marcinfourcyclefree2014} and Marcin Wroncha for further discussions regarding the four-cycle hom-free condition. He would like to thank Prof. Yuval Peres for introducing him to \cite{norrisdiffuse2015}. He would like to thank Prof. Peter Winkler for hosting him and discussing the open problems as stated in Section \ref{section: open question}. He would also like to thank Prof. Lior Silberman, Prof. Omer Angel for many lively discussions and Raimundo Brice\~no, Prof. Michael Schraudner and Prof. Ronnie Pavlov for giving a patient ear to his ideas. Part of this research was funded by ERC starting grant LocalOrder, Israel Science Foundation grant 861/15, 7599/13 and ICERM. Both the authors would like to thank the careful referee whose comments greatly improved the quality of the exposition.
\bibliographystyle{abbrv}
\bibliography{blockgluing}

\begin{thebibliography}{10}

\bibitem{Angluin80}
D.~Angluin.
\newblock Local and global properties in networks of processors (extended
  abstract).
\newblock In {\em Proceedings of the 12th Annual {ACM} Symposium on Theory of
  Computing, April 28-30, 1980, Los Angeles, California, {USA}}, pages 82--93,
  1980.

\bibitem{bergerundecidable}
R.~Berger.
\newblock The undecidability of the domino problem.
\newblock {\em Mem. Amer. Math. Soc. No.}, 66:72, 1966.

\bibitem{norrisdiffuse2015}
E.~Boissard, S.~Cohen, T.~Espinasse, and J.~Norris.
\newblock Diffusivity of a random walk on random walks.
\newblock {\em Random Structures Algorithms}, 47(2):267--283, 2015.

\bibitem{boweneqlecturenotes}
R.~Bowen.
\newblock {\em Equilibrium states and the ergodic theory of {A}nosov
  diffeomorphisms}, volume 470 of {\em Lecture Notes in Mathematics}.
\newblock Springer-Verlag, Berlin, revised edition, 2008.
\newblock With a preface by David Ruelle, Edited by Jean-Ren{\'e} Chazottes.

\bibitem{boyle2010multidimensional}
M.~Boyle, R.~Pavlov, and M.~Schraudner.
\newblock Multidimensional sofic shifts without separation and their factors.
\newblock {\em Trans. Amer. Math. Soc.}, 362(9):4617--4653, 2010.

\bibitem{Rai4cb3d}
R.~Brice\~no.
\newblock Personal communication, 2014.

\bibitem{Raimundo2014}
R.~Brice{\~n}o.
\newblock The topological strong spatial mixing property and new conditions for
  pressure approximation.
\newblock http://arxiv.org/abs/1411.2289, 2014.

\bibitem{Raimundopavlov2016}
R.~Brice{\~n}o and R.~Pavlov.
\newblock Strong spatial mixing in homomorphism spaces.
\newblock http://arxiv.org/abs/1510.01453, 2015.

\bibitem{brightwell2000gibbs}
G.~R. Brightwell and P.~Winkler.
\newblock Gibbs measures and dismantlable graphs.
\newblock {\em J. Combin. Theory Ser. B}, 78(1):141--166, 2000.

\bibitem{burtonsteiffnonuniquesft}
R.~Burton and J.~E. Steif.
\newblock Non-uniqueness of measures of maximal entropy for subshifts of finite
  type.
\newblock {\em Ergodic Theory Dynam. Systems}, 14(2):213--235, 1994.

\bibitem{MR3645513}
N.~Chandgotia.
\newblock Four-cycle free graphs, height functions, the pivot property and
  entropy minimality.
\newblock {\em Ergodic Theory Dynam. Systems}, 37(4):1102--1132, 2017.

\bibitem{PTEdurret}
R.~Durrett.
\newblock {\em Probability: theory and examples}.
\newblock Cambridge Series in Statistical and Probabilistic Mathematics.
  Cambridge University Press, Cambridge, fourth edition, 2010.

\bibitem{hatchertopo}
A.~Hatcher.
\newblock {\em Algebraic topology}.
\newblock Cambridge University Press, Cambridge, 2002.

\bibitem{lieb}
E.~Lieb.
\newblock Residual entropy of square ice.
\newblock {\em Physical Review}, 162:162 -- 172, 1967.

\bibitem{LM}
D.~Lind and B.~Marcus.
\newblock {\em An introduction to symbolic dynamics and coding}.
\newblock Cambridge University Press, 1995, reprinted 1999.

\bibitem{Nowwinkler}
R.~Nowakowski and P.~Winkler.
\newblock Vertex-to-vertex pursuit in a graph.
\newblock {\em Discrete Math.}, 43(2-3):235--239, 1983.

\bibitem{schraudnerpavlovblockgluingdist}
R.~Pavlov and M.~Schraudner.
\newblock Personal communication, 2015.

\bibitem{Robinson1971}
R.~M. Robinson.
\newblock Undecidability and nonperiodicity for tilings of the plane.
\newblock {\em Invent. Math.}, 12:177--209, 1971.

\bibitem{Rue}
D.~Ruelle.
\newblock {\em Thermodynamic formalism}.
\newblock Cambridge Mathematical Library. Cambridge University Press,
  Cambridge, second edition, 2004.
\newblock The mathematical structures of equilibrium statistical mechanics.

\bibitem{schmidt_cohomology_SFT_1995}
K.~Schmidt.
\newblock The cohomology of higher-dimensional shifts of finite type.
\newblock {\em Pacific J. Math.}, 170(1):237--269, 1995.

\bibitem{schimdt_fund_cocycle_98}
K.~Schmidt.
\newblock Tilings, fundamental cocycles and fundamental groups of symbolic
  {${\Bbb Z}^d$}-actions.
\newblock {\em Ergodic Theory Dynam. Systems}, 18(6):1473--1525, 1998.

\bibitem{Stallingsgraph1983}
J.~R. Stallings.
\newblock Topology of finite graphs.
\newblock {\em Invent. Math.}, 71(3):551--565, 1983.

\bibitem{Marcinfourcyclefree2014}
M.~Wrochna.
\newblock {Homomorphism Reconfiguration via Homotopy}.
\newblock In E.~W. Mayr and N.~Ollinger, editors, {\em 32nd International
  Symposium on Theoretical Aspects of Computer Science (STACS 2015)}, volume~30
  of {\em Leibniz International Proceedings in Informatics (LIPIcs)}, pages
  730--742, Dagstuhl, Germany, 2015. Schloss Dagstuhl--Leibniz-Zentrum fuer
  Informatik.

\end{thebibliography}

\end{document}